\documentclass[12pt,twoside,a4paper]{amsart}
\usepackage[centertags]{amsmath}
\usepackage{amssymb,amscd,amsfonts,latexsym,a4wide,amsthm}

\advance\headheight by 3pt

\renewcommand{\subjclass}[1]{\thanks{\emph{2010 Mathematics Subject Classification:}~#1}}
\renewcommand{\keywords}[1]{\thanks{\emph{Keywords and Phrases:}~#1}}
\renewcommand{\date}{\thanks{\today}}

\newtheorem{theorem}{Theorem}[section]
\newtheorem{lemma}{Lemma}[section]
\newtheorem{corollary}{Corollary}[section]
\newtheorem{proposition}[lemma]{Proposition}
\numberwithin{equation}{section}

\def\teto#1{\setbox\z@\hbox{${#1\vphantom k}$}\hbox{%
 \hbox{\lower2\ex@\hbox{\lower\dp\z@\hbox{\vbox{\hrule
 \hbox{\vrule\hskip2\ex@\vbox{\vskip2\ex@\box\z@\vskip1\ex@}%
 \hskip2\ex@\vrule}}}}}}}
\catcode`\@=\active

\newcommand{\s}{\mathcal{S}}

\newcommand{\Q}{\mathbb{Q}}
\newcommand{\oQ}{\overline{\mathbb{Q}}}
\newcommand{\Z}{\mathbb{Z}}

\newcommand{\K}{\Bbbk}
\newcommand{\obK}{\overline{\Bbbk}}

\newcommand{\Qq}{\mathbb{Q}}
\newcommand{\OQq}{\overline{\mathbb{Q}}}
\newcommand{\Zz}{\mathbb{Z}}

\newcommand{\x}{\mathbf{x}}

\newcommand{\av}{\mathbf{a}}

\newcommand{\uv}{\mathbf{u}}

\newcommand{\OO}{\mathcal{O}}

\newcommand{\M}{\mathcal{M}}

\newcommand{\F}{\mathcal{F}}

\newcommand{\HH}{\mathcal{H}}

\newcommand{\hh}{\hat{h}}

\newcommand{\vp}{\varphi}

\newcommand{\al}{\alpha}
\newcommand{\be}{\beta}

\newcommand{\de}{\delta}
\newcommand{\oh}{\overline{h}}
\newcommand{\odeg}{\overline{\deg}\,}

\newcommand{\fp}{\frak{p}}
\newcommand{\fa}{\frak{a}}

\newcommand{\til}[1]{\tilde{#1}}

\newcommand{\kdots}{,\ldots ,}
\renewcommand{\eqref}[1]{(\ref{#1})}

\title[Diophantine equations over finitely generated domains]{Effective results for Diophantine equations over finitely generated
domains}

\author[A. B\'erczes]{Attila B\'erczes}

\thanks{The research was supported in part by the Hungarian Academy of
Sciences, and by grants K100339 (A.B., K.G.), NK104208 (A.B., K.G.)
and K75566 (A.B.) of the Hungarian National Foundation for
Scientific Research. The work is supported
by the T\'AMOP 4.2.1./B-09/1/KONV-2010-0007 project. The project
is implemented through the New Hungary Development Plan,
co-financed by the European Social Fund and the European Regional
Development Fund.}
\address{A. B\'erczes \newline
         \indent Institute of Mathematics, University of Debrecen \newline
         \indent H-4010 Debrecen, P.O. Box 12, Hungary}
\email{berczesa\char'100science.unideb.hu}

\author[J.-H. Evertse]{Jan-Hendrik Evertse}
\address{J.-H. Evertse \newline
         \indent Universiteit Leiden, Mathematisch Instituut, \newline
         \indent Postbus 9512, 2300 RA Leiden, The Netherlands}
\email{evertse\char'100math.leidenuniv.nl}

\author[K. Gy\H{o}ry]{K\'{a}lm\'{a}n Gy\H{o}ry}
\address{K. Gy\H{o}ry \newline
         \indent Institute of Mathematics, University of Debrecen \newline
         \indent H-4010 Debrecen, P.O. Box 12, Hungary}
\email{gyory\char'100science.unideb.hu}

\begin{document}
\maketitle

\section{Introduction.}\label{S_Int}

Let $A$ be an arbitrary integral domain of characteristic $0$ that is finitely
generated over $\Zz$.
We consider Thue equations $F(x,y)=\delta$ in $x,y\in A$,
where $F$ is a binary form with coefficients from $A$ and $\delta$ is
a non-zero element from $A$, and hyper- and superelliptic equations
$f(x)=\delta y^m$ in $x,y\in A$, where $f\in A[X]$,
$\delta\in A\setminus\{ 0\}$ and $m\in \Z_{\geq 2}$.

Under the necessary finiteness conditions we give effective upper bounds for the sizes (defined in Section \ref{S_Results}) of the
solutions of the equations in terms of appropriate
representations for $A$, $\delta$, $F$, $f$, $m$.
These results imply that the solutions of these equations can be determined in principle.
Further, we consider the Schinzel-Tijdeman equation
$f(x)=\delta y^m$ where $x,y\in A$ and $m\in\Zz_{\geq 2}$ are the unknowns
and give an effective upper bound for $m$.

We mention that results from the existing literature deal only with
equations over restricted classes of finitely generated domains
whereas we do not have to impose any restrictions on $A$.
Further, our upper bounds for the sizes of the solutions $x,y$
and $m$ are new, also for the special cases considered earlier.
Our proofs are a combination of existing effective results for Thue equations
and hyper- and superelliptic equations over number fields and over function
fields, and a recent effective specialization method of
Evertse and Gy\H{o}ry \cite{EGy9}.

We give a brief overview of earlier results.
A major breakthrough in the effective theory of Diophantine
equations was established by A. Baker in the 1960's.
Using his own estimates for linear forms in logarithms of
algebraic numbers,
he obtained effective finiteness results, i.e.,
with explicit upper bounds for the absolute values of the solutions, for
Thue equations \cite{Baker67}
and hyper- and superelliptic equations
\cite{Baker69} over $\Zz$.
Schinzel and Tijdeman \cite{SchinzelTijdeman76} were the first to consider superelliptic equations $f(x)=\delta y^m$ over $\Zz$
where also the exponent $m$ was taken as an unknown
and gave an effective upper bound for $m$.
Their proof also depends on Baker's linear forms estimates.

The effective results of Baker and of Schinzel and Tijdeman were extended to
equations where
the solutions $x,y$ are taken from larger integral domains; we mention here
Coates \cite{Coates68}, Sprind\v{z}uk and Kotov \cite{SprinKotov73}
(Thue equations over
$\OO_S$, where $\OO_S$ is the ring of $S$-integers of an algebraic
number field), Trelina \cite{Trelina78}, Brindza \cite{B6} (hyper- and
superelliptic equations over $\OO_S$),
Gy\H{o}ry \cite{Gy3} (Thue equations over a restricted class of integral
domains finitely generated over $\Zz$ that contain transcendental
elements), Brindza \cite{B5} and V\'{e}gs\H{o} \cite{Vegso94}
(hyper- and superelliptic equations and the Schinzel-Tijdeman equation
over the class of domains considered by Gy\H{o}ry).
These last mentioned works of Gy\H{o}ry, Brindza and V\'{e}gs\H{o}
were based on an effective specialization method developed
by Gy\H{o}ry in the 1980's \cite{Gy3}, \cite{Gy20}.

Recently,
Evertse and Gy\H{o}ry \cite{EGy9}
extended Gy\H{o}ry's specialization method so that it can now be used
to prove effective results for Diophantine equations over arbitrary finitely generated domains $A$
over $\Zz$, without any further restriction on $A$ whatsoever.
They applied this to
unit equations $ax+by=c$ in units $x,y$ of $A$,
and gave an effective upper bound for the
sizes of the solutions $x,y$
in terms of appropriate representations for $A,a,b,c$.
In their method of proof, Evertse and Gy\H{o}ry used existing effective
results for $S$-unit equations over number fields and function fields,
and combined these with their general specialization method.

The approach of Evertse and Gy\H{o}ry can be applied to various other classes
of Diophantine equations.
In the present paper, we have worked out the consequences for Thue equations, hyper-and superelliptic equations, and Schinzel-Tijdeman equations.

\section{Results}\label{S_Results}

We first introduce the necessary notation and then state our results.

\subsection{Notation}\label{SS_notation}
Let $A=\Z[z_1, \dots , z_r]$ be a finitely generated integral domain
of characteristic $0$ which is finitely generated over $\Z$. We assume that $r>0$. We have
$$
A \cong \Z[X_1, \dots, X_r]/I
$$
where $I$ is the ideal of polynomials $f \in \Z[X_1, \dots, X_r]$ such that $f(z_1, \dots, z_r)=0$. The ideal
$I$ is finitely generated, say
$$
I=(f_1\kdots f_t).
$$
We may view $f_1\kdots f_t$ as a representation for $A$.
Recall that a necessary and sufficient condition for $A$ to be a domain
of characteristic zero is that $I$ be a prime ideal with $I\cap\Zz =(0)$.
Given a set of generators $\{f_1\kdots f_t\}$ for $I$ this can be checked
effectively (see for instance Aschenbrenner
\cite[Cor. 6.7, Lemma 6.1]{Aschenbrenner04} but this
follows already from work of Hermann \cite{Hermann26}).

Denote by $K$ the quotient field of $A$. For $\al \in A$, we call $f$  a {\it representative} for $\al$, or
we say that $f$ represents $\al$, if $f \in \Z[X_1, \dots, X_r]$ and
$\al=f(z_1, \dots , z_r)$. Further,
for $\al\in K$ we call $(f,g)$ a {\it pair of representatives} for $\al$, or say that $(f,g)$ represents $\al$ if
$f,g \in \Z[X_1, \dots, X_r]$, $g \not\in I$ and
$\al=f(z_1, \dots , z_r)/g(z_1, \dots , z_r)$.

Using an ideal membership algorithm
for $\Zz [X_1\kdots X_r]$ (see e.g., Aschenbrenner \cite[Theorem A]{Aschenbrenner04}
but such algorithms were probably known in the 1960's), one can decide effectively whether two
polynomials $f',f''\in\Zz [X_1\kdots X_r]$ represent the same element
of $A$, i.e., $f'-f''\in I$, or whether two pairs of polynomials
$(f',g'),(f'',g'')$ in $\Zz [X_1\kdots X_r]$
represent the same element of $K$, i.e., $g'\not\in I$, $g''\not\in I$
and $f'g''-f''g'\in I$.

Given a non-zero polynomial $f\in\Zz [X_1\kdots X_r]$,
we denote by $\deg f$ its total degree and by
$h(f)$ its logarithmic height, that is the logarithm of the maximum of the absolute values
of its coefficients. Then the {\it size} of $f$ is defined by
$$
s(f):=\max(1,\deg f, h(f)).
$$
Further, we define $s(0):=1$.
It is clear that there are only finitely many polynomials in $\Z[X_1, \dots, X_r]$ of size below a given bound, and
these can be determined effectively.

Throughout the paper we shall use the notation $O(\cdot)$ to denote a quantity which is $c$ times the expression between the parentheses,
where $c$ is an effectively computable positive absolute constant which may be different at each occurrence of the $O$-symbol. Further,
throughout the paper we write
$$
\log^* a :=\max(1, \log a) \quad \text{for $a>0$}, \quad\quad \log^* 0:=1.
$$

\subsection{Thue equations}\label{SS_results_Thue}
We consider the Thue equation over $A$,
\begin{equation}\label{Thue}
F(x,y)=\de \quad\quad \text{in} \quad x,y \in A,
\end{equation}
where
$$
F(X,Y)=a_0X^n+a_1X^{n-1}Y+\dots+a_nY^n \in A[X,Y]
$$
is a binary form of degree $n\geq 3$ with discriminant $D_F\not= 0$, and $\de \in A\setminus \{ 0 \}$. Choose
representatives
$$
\til{a_0}, \til{a_1}, \dots , \til{a_n}, \til{\de}\in \Z[X_1, \dots , X_r]
$$
of $a_0, a_1, \dots , a_n, \de$, respectively.
To ensure that $\de\not= 0$ and $D(F)\not= 0$, we have to choose the
representatives in such a way that $\til{\de}\not\in I$, $D_{\til{F}}\not\in I$
where $D_{\til{F}}$ is the discriminant of
$\til{F}:=\sum_{j=0}^n \til{a_j}X^{n-j}Y^j$. These last two conditions
can be checked by means of the ideal membership algorithm
mentioned above.
Let
\begin{equation}\label{Thue_h_deg_assumption}
\begin{cases}
&\max (\deg f_1, \dots , \deg f_t, \deg \til{a_0}, \deg \til{a_1}, \dots , \deg \til{a_n}, \deg \til{\de}) \leq d \\
&\max (h(f_1), \dots ,h(f_t), h(\til{a_0}), h(\til{a_1}), \dots , h(\til{a_n}), h(\til{\de})) \leq h,
\end{cases}
\end{equation}
where $d\geq 1$, $h\geq 1$.

\begin{theorem}\label{T_Thue}
Every solution $x,y$ of equation \eqref{Thue} has representatives $\til{x}, \til{y}$ such that
\begin{equation}\label{Thue_boundsintheorem}
s(\til{x}), s(\til{y}) \leq \exp \left(n!(nd)^{\exp O(r)}(h+1) \right).
\end{equation}
\end{theorem}
The exponential dependence of the upper bound on $n!$, $d$ and $h+1$ is coming from
a Baker-type effective result for Thue equations over number fields
that is used in the proof.
The bad dependence on $r$ is coming from the
effective commutative
algebra for polynomial rings over fields and over $\Zz$,
that is used in the specialization method of Evertse and Gy\H{o}ry mentioned
above.

We immediately deduce that equation \eqref{Thue} is effectively solvable:

\begin{corollary}\label{C_Thue}
There exists an algorithm which,
for any given $f_1\kdots f_t$ such that $A$
is a domain, and any representatives $\til{a_0}\kdots\til{a_n}$, $\tilde{\de}$
such that $D_{\til{F}},\til{\de}\not\in I$, computes
a finite list, consisting of one pair of representatives for each solution $(x,y)$
of \eqref{Thue}.
\end{corollary}

\begin{proof}
Let $C$ be the upper bound from \eqref{Thue_boundsintheorem}.
Check for each pair of polynomials $\til{x},\til{y}\in\Zz [X_1\kdots X_r]$
of size at most $C$ whether $\til{F}(\til{x},\til{y})-\til{\de}\in I$.
Then for all pairs $\tilde{x},\tilde{y}$ passing this test,
check whether they are equal modulo $I$, and keep a maximal subset of pairs
that are different modulo $I$.
\end{proof}

\subsection{Hyper- and superelliptic equations}\label{SS_results_hypersuper}
We now consider the equation
\begin{equation}\label{hypersuper}
F(x)=\de y^m \quad\quad \text{in} \quad x,y \in A,
\end{equation}
where
$$
F(X)=a_0X^n+a_1X^{n-1}+\dots+a_n \in A[X]
$$
is a polynomial of degree $n$ with discriminant $D_F \ne 0$, and where $\de \in A\setminus \{ 0 \}$. We assume that either $m=2$ and $n\geq 3$, or $m\geq 3$
and $n\geq 2$. For $m=2$, equation \eqref{hypersuper} is called
a \emph{hyperelliptic equation}, while for $m\geq 3$ it is called
a \emph{superelliptic equation}. Choose again
representatives
$$
\til{a_0}, \til{a_1}, \dots , \til{a_n}, \til{\de}\in \Z[X_1, \dots , X_r]
$$
for $a_0, a_1, \dots , a_n, \de$, respectively.
To guarantee that $\delta\not= 0$ and $D_F\not= 0$, we have to choose the
representatives in such a way that
$\til{\de}$ and the discriminant of $\til{F}:=\sum_{j=0}^n \til{a_j}X^{n-j}$ do not belong to $I$.
Let
\begin{equation}\label{hypersuper_h_deg_assumption}
\begin{cases}
&\max (\deg f_1, \dots , \deg f_t, \deg \til{a_0}, \deg \til{a_1}, \dots , \deg \til{a_n}, \deg \til{\de}) \leq d \\
&\max (h(f_1), \dots ,h(f_t), h(\til{a_0}), h(\til{a_1}), \dots , h(\til{a_n}), h(\til{\de})) \leq h,
\end{cases}
\end{equation}
where $d\geq 1$, $h\geq 1$.

\begin{theorem}\label{T_hypersuper}
Every solution $x,y$ of equation \eqref{hypersuper} has representatives $\til{x}, \til{y}$ such that
\begin{equation}\label{hypersuper_boundsintheorem}
s(\til{x}), s(\til{y}) \leq  \exp \left( m^3(nd)^{\exp{O(r)}}(h+1) \right).
\end{equation}
\end{theorem}

Completely similarly as for Thue equations, one can determine effectively
a finite list, consisting of one pair of representatives
for each solution $(x,y)$ of \eqref{hypersuper}.

Our next result deals with the Schinzel-Tijdeman equation,
which is \eqref{hypersuper} but with three unknowns $x,y\in A$
and $m\in\Zz_{\geq 2}$.

\begin{theorem}\label{T_ST_over_A}
Assume that in \eqref{hypersuper}, $F$ has non-zero discriminant and $n\geq 2$.
Let $x,y\in A$, $m\in\Zz_{\geq 2}$ be a solution of \eqref{hypersuper}. Then
\begin{eqnarray}
\label{bound_ST_over_A}
&&m\leq \exp \left( (nd)^{\exp O(r)} (h+1) \right)
\\
\nonumber
&&\qquad\quad
\mbox{if }
y\in\OQq,\ y\not=0,\ y\ \mbox{is not a root of unity,}
\\[0.2cm]
\label{bound_ST_over_A-alg}
&&m\leq (nd)^{\exp O(r)}\ \ \mbox{if } y\not\in\OQq .
\end{eqnarray}
\end{theorem}

\section{A reduction}\label{S_reduction}

We shall reduce our equations to equations of the same type over an integral domain $B\supseteq A$ of a special type which is more convenient
to deal with.

As before, let $A=\Zz [z_1\kdots z_r]$ be an integral domain
which is finitely generated over $\Zz$ and let $K$ be the quotient field of $A$.
Suppose that $K$ has transcendence degree $q\geq 0$. If $q>0$, we assume without loss of generality that
$\{z_1, \dots, z_q \}$ forms a transcendence basis of $K/\Q$. Write $\rho:=r-q$. We define
$$
\begin{aligned}
&A_0:=\Z[z_1, \dots, z_q], \quad &&K_0:=\Q(z_1, \dots , z_q) \quad && \text{if $q>0$} \\
&A_0:=\Z, \quad &&K_0 :=\Q && \text{if $q=0$}.
\end{aligned}
$$
The field $K$ is a finite extension of $K_0$. Further, if $q=0$, it is an algebraic number field.
In case that $q>0$, for $f\in A_0\setminus\{ 0\}$ we define $\deg f$
and $h(f)$ to be the total degree and logarithmic height of $f$, viewed
as a polynomial in the variables $z_1\kdots z_q$.
In case that $q=0$, for $f\in A_0\setminus\{ 0\}=\Z\setminus\{ 0\}$,
we put $\deg f := 0$ and $h(f):=\log |f|$. 

We shall construct an integral extension $B$ of $A$ in $K$ such that
\begin{equation}\label{extended_domain}
B:=A_0[w,f^{-1}],
\end{equation}
where $f\in A_0$ and $w$ is a primitive element of $K$ over $K_0$ which is integral over $A_0$. Then we give a bound for the sizes of the
solutions of our equations in $x,y\in B$.

We recall that $A\cong\Zz [X_1\kdots X_r]/I$ where
$I\subset \Zz [X_1\kdots X_r]$
is the ideal of polynomials $f$ with $f(z_1\kdots z_r)=0$
and $z_i$ corresponds
to the residue class of $X_i$ modulo $I$.
The ideal $I$ is finitely generated.  Assume that
$$
I=(f_1, \dots , f_t),
$$
and put
\begin{equation}
d_0:=\max (1,\deg f_1, \dots , \deg f_t),\quad\quad h_0:=\max (1,h(f_1), \dots ,h(f_t)).
\end{equation}

\begin{proposition}\label{P_newdomaingenerators}
(i) There is a $w\in A$ such that $K=K_0(w)$, $w$ is integral over $A_0$ and $w$ has minimal polynomial
$$
\F(X)=X^D+\F_1 X^{D-1}+\dots+\F_D\in A_0[X]
$$
over $K_0$ such that $D\leq d_0^{\rho}$ and
\begin{equation}\label{bound_minpol_fieldgen}
\deg \F_k \leq (2d_0)^{\exp O(r)}, \quad\quad h(\F_k) \leq (2d_0)^{\exp O(r)}(h_0+1)
\end{equation}
for $k=1, \dots, D$.

(ii) Let $\al_1, \dots, \al_k\in K^*$ and suppose that the pairs $u_i, v_i\in \Z[X_1, \dots, X_r]$, $v_i \not\in I$ represent $\al_i$ for $i=1,\dots, k$, respectively. Put
$$
\begin{aligned}
&d^{**}:=\max(d_0, \deg u_1, \deg v_1, \dots, \deg u_k, \deg v_k), \\
&h^{**}:=\max(h_0, h(u_1), h(v_1), \dots, h(u_k), h(v_k)).
\end{aligned}
$$
Then there is a non-zero $f\in A_0$ such that
\begin{equation}\label{extended_domain_III}
\begin{aligned}
&A \subseteq A_0[w,f^{-1}], \\
&\al_1, \dots, \al_k\in A_0[w,f^{-1}]^* \\
\end{aligned}
\end{equation}
and
\begin{equation}\label{bound_minpol_fieldgen_II}
\deg f \leq (k+1)(2d^{**})^{\exp O(r)}, \quad\quad h(f) \leq (k+1) (2d^{**})^{\exp O(r)}(h^{**}+1).
\end{equation}
\end{proposition}

\begin{proof}
For (i) see Evertse and Gy\H ory \cite{EGy9}, Proposition 3.4 and Lemma 3.2, (i), and for (ii) see \cite{EGy9}, Lemma 3.6.
\end{proof}

We shall use Proposition \ref{P_newdomaingenerators}, (ii) in a special case. To state it, we introduce some
further notation and prove a lemma.

We recall that $a_0, a_1, \dots, a_n \in A$ are the coefficients of the binary form $F(X,Y)$, resp. of the polynomial $F(X)$
in Sections \ref{SS_results_Thue} resp. \ref{SS_results_hypersuper}, and $\til{a}_0, \til{a}_1, \dots, \til{a}_n$ denote their
representatives satisfying \eqref{Thue_h_deg_assumption} resp. \eqref{hypersuper_h_deg_assumption}.
This implies that $d_0\leq d$, $h_0\leq h$, and that $\til{a}_i$
has total degree $\leq d$ and logarithmic height $\leq h$ for $i=0\kdots n$.
Denote by $\til{F}$ the binary form $F(X,Y)$ resp. the polynomial $F(X)$ with coefficients $a_0, a_1, \dots, a_n$ replaced by $\til{a}_0, \til{a}_1, \dots, \til{a}_n$, and by $D_{\til{F}}$ the discriminant of $\til{F}$. In view of the assumption $D_F\ne 0$ we have $D_{\til{F}}\not\in I$.

Keeping the notation and assumptions of Sections \ref{SS_results_Thue} resp. \ref{SS_results_hypersuper}, we have the following lemma.

\begin{lemma}\label{L_bound_h_deg_DFtil}
For the discriminant $D_{\til{F}}$ the following statements are true:
\begin{align}
\label{bound_deg_DFtil} &\deg D_{\til{F}} \leq (2n-2)d, \\
\label{bound_h_DFtil} &h(D_{\til{F}}) \leq (2n-2)\left( \log \left( 2n^2 \binom{d+r}{r} \right) +h \right).
\end{align}
\end{lemma}

\begin{proof}
Recall that the discriminant $D_{\til{F}}$ can be expressed as
\begin{equation}\label{determinant-expression}
D(\til{F})=\pm
\left|
\begin{array}{ccccccc}
\til{a}_0 &\til{a}_1      &\cdots&\cdots &\til{a}_n     &      &       \\
    &\ \ddots &      &       &        &\ddots&       \\
    &         &\til{a}_0   &\til{a}_1    &\cdots  &\cdots&\til{a}_m    \\
\til{a}_1 &2\til{a}_2     &\cdots&n\til{a}_n   &        &      &       \\
n\til{a}_0&(n-1)\til{a}_1 &\cdots&\til{a}_{n-1}&        &      &       \\
    &\ \ddots &      &       &\ddots  &      &       \\
    &         &\ddots&       &        &\ddots&       \\
    &         &      &n\til{a}_0   &(n-1)\til{a}_1&\cdots&\til{a}_{n-1}\\
\end{array}
\right|\, ,
\end{equation}
with on the first $n-2$ rows of the determinant $\til{a_0}\kdots \til{a_n}$,
on the $(n-1)$-st row $\til{a}_1,2\til{a}_2\kdots n\til{a}_n$, and on the last $n-1$ rows
$n\til{a}_0\kdots \til{a}_{n-1}$.
This implies at once \eqref{bound_deg_DFtil}.

To prove \eqref{bound_h_DFtil}, we use the length $L(P)$ of a polynomial $P\in \Z[X_1, \dots ,X_r]$, that is the sum of the
absolute values of the coefficients of $P$. It is known and easily seen that if $P,Q\in \Z[X_1, \dots , X_r]$ then $L(P+Q)$ and
$L(PQ)$ do not exceed $L(P)+L(Q)$ and $L(P)L(Q)$, respectively (see e.g. Waldschmidt \cite{Waldschmidt2000}, p.76).

We have
$$
L(\til{a}_i)\leq \binom{d+r}{r}H \quad\quad \text{with} \quad H=\exp h \quad \quad \text{for $i=0,\dots , n$}.
$$
By applying these facts to \eqref{determinant-expression},
we obtain
$$
L(D_{\til{F}}) \leq (2n-2)! \left(n  \binom{d+r}{r}H \right)^{2n-2}.
$$
Together with $h(D_{\til{F}}) \leq \log L(D_{\til{F}})$ this implies \eqref{bound_h_DFtil}.
\end{proof}

We now apply Proposition \ref{P_newdomaingenerators}, (ii) to the numbers $\al_1=\de$, $\al_2=\de^{-1}$, $\al_3=D_F$ and
$\al_4=D_F^{-1}$.
Then the pairs $(\til{\de},1)$, $(1,\til{\de})$, $(D_{\til{F}},1)$, $(1,D_{\til{F}})$ represent the numbers
$\al_i$, $i=1, \dots ,4$. Using the upper bounds for
$\deg D_{\til{F}}$,  $h(D_{\til{F}})$ implied by
Lemma \ref{L_bound_h_deg_DFtil}
as well as the upper bounds $\deg\til{\de}\leq d$, $h(\til{\de})\leq h$
implied by \eqref{Thue_h_deg_assumption},  \eqref{hypersuper_h_deg_assumption},
we get immediately from Proposition \ref{P_newdomaingenerators}, (ii) the following.

\begin{proposition}\label{P_newdomaingenerators_spec_case}
There is a non-zero $f\in A_0$ such that
\begin{equation}\label{extended_domain_IV}
A \subseteq A_0[w,f^{-1}],\ \ \
\de, D_F \in A_0[w,f^{-1}]^*
\end{equation}
and
\begin{equation}\label{bound_minpol_fieldgen_III}
\deg f \leq (nd)^{\exp O(r)}, \quad\quad h(f) \leq (nd)^{\exp O(r)}(h+1).
\end{equation}
\end{proposition}

In the case $q>0$, $z_1, \dots , z_q$ are algebraically independent. Thus, for $q\geq 0$, $A_0$ is a unique factorization domain, and hence the
greatest common divisor of a finite set of elements of $A_0$ is well defined and up to sign uniquely determined. We associate with every element $\al\in K$ the
up to sign unique tuple $P_{\al,0}, \dots, P_{\al, D-1}, Q_{\al}$ of elements of $A_0$ such that
\begin{equation}\label{representation_typeII}
\al=Q_{\al}^{-1}\sum_{j=0}^{D-1}P_{\al,j}w^j \quad \text{with} \quad Q_{\al}\ne 0, \ \ \gcd(P_{\al,0}, \dots, P_{\al, D-1}, Q_{\al})=1.
\end{equation}
We put
\begin{equation}\label{degbar_hbar}
\begin{cases}
&\odeg \al:=\max(\deg P_{\al,0}, \dots, \deg P_{\al, D-1}, \deg Q_{\al}) \\
&\oh(\al):=\max(h(P_{\al,0}), \dots, h(P_{\al, D-1}), h(Q_{\al})),
\end{cases}
\end{equation}
where as usual, $\deg P$, $h(P)$ denote the total degree and logarithmic height
of a polynomial $P$ with rational integral coefficients.
Thus for $q=0$ we have $\odeg \al=0$ and $\oh(\al)=\log\max(|P_{\al,0}|, \dots, |P_{\al, D-1}|, |Q_{\al}|)$.

\begin{lemma}\label{L_degbar_hbar_by_d*_h*}
Let $\al\in K^*$ and let $(a,b)$ be a pair of representatives for $\al$ with $a,b\in \Z[X_1, \dots , X_r]$, $b \not\in I$.
Put
$$
d^*=\max(d_0, \deg a, \deg b) \quad \text{and} \quad h^*:=\max(h_0,h(a),h(b)).
$$
Then
\begin{equation}\label{bound_degbar_hbar_by_d*_h*}
\odeg \al \leq (2d^*)^{\exp O(r)}, \quad\quad \oh(\al)\leq (2d^*)^{\exp O(r)}(h^*+1).
\end{equation}
\end{lemma}

\begin{proof}
This is Lemma 3.5 in Evertse and Gy\H ory \cite{EGy9}.
\end{proof}

\begin{lemma}\label{L_boundedrep}
Let $\al$ be a nonzero element of $A$, and put
$$
\widehat{d}:=\max(d_0, \odeg \al), \quad\quad \widehat{h}:=\max(h_0, \oh(\al)).
$$
Then $\al$ has a representative $\til{\al}\in \Z[X_1, \dots , X_r]$ such that
\begin{equation}\label{bound_deg_h_by_dtil_htil}
\left\{\begin{array}{l}
\deg \til{\al} \leq (2\widehat{d})^{\exp O(r\log^* r)}(\widehat{h}+1),
\\
h(\til{\al})\leq (2\widehat{d})^{\exp O(r\log^* r)}(\widehat{h}+1)^{r+1}.
\end{array}\right.
\end{equation}
\end{lemma}

\begin{proof}
This is a special case of Lemma 3.7 of Evertse and Gy\H ory \cite{EGy9} with the choice $\lambda=1$ and $a=b=1$. The proof of this lemma
is based on work of Aschenbrenner \cite{Aschenbrenner04}.
\end{proof}

\subsection{Thue equations}

Recall that $A_0=\Z[z_1, \dots , z_q]$, $K_0=\Q(z_1, \dots , z_q)$ if $q>0$, and $A_0=\Z$, $K_0=\Q$ if $q=0$, and that in the case $q=0$
total degrees and $\odeg$-s are always zero. Further, we have
$$
F(X,Y)=a_0X^n+a_1X^{n-1}Y+\dots+a_nY^n \in A[X,Y]
$$
with $n\geq 3$ and with discriminant $D_F \ne 0$, and $\de \in A\setminus \{ 0 \}$. Recall that for
$a_0, a_1, \dots , a_n, \de$ we have chosen representatives $\til{a_0}, \til{a_1}, \dots , \til{a_n}, \til{\de}\in \Z[X_1, \dots , X_r]$
satisfying \eqref{Thue_h_deg_assumption}.

Theorem \ref{T_Thue} will be deduced from the following Proposition, which makes sense also if $q=0$. The proof of this proposition is given
in Sections \ref{S_boundthedegree}--\ref{S_boundtheheight}.

\begin{proposition}\label{P_Thue_over_B}
Let $w$ and $f$ be as in Propositions \ref{P_newdomaingenerators}, (i) and \ref{P_newdomaingenerators_spec_case}, respectively, with the properties specified there, and consider the integral domain
$$
B:=A_0[f^{-1}, w].
$$
Then for the solutions $x,y$ of the equation
\begin{equation} \label{Thue_over_B}
F(x,y)=\de \quad\quad \text{in} \quad x,y\in B
\end{equation}
we have
\begin{align}
\label{Thue_degbound_B} &\odeg x, \  \odeg y \leq (nd)^{\exp O(r)},\\
\label{Thue_hbound_B} &\oh(x),\ \oh(y) \leq \exp \left(n!(nd)^{\exp{O(r)}}(h+1) \right).
\end{align}
\end{proposition}

We now deduce Theorem \ref{T_Thue} from Proposition \ref{P_Thue_over_B}.

\begin{proof}[Proof of Theorem \ref{T_Thue}]
Let $x,y$ be a solution of equation \eqref{Thue}. In view of \eqref{extended_domain_IV} $x,y$ is also a solution
in $B=A_0[f^{-1},w]$, where $f,w$ satisfy the conditions specified in Propositions \ref{P_newdomaingenerators}, (i) and \ref{P_newdomaingenerators_spec_case}, respectively. Then by
Proposition \ref{P_Thue_over_B}, the inequalities  \eqref{Thue_degbound_B} and \eqref{Thue_hbound_B} hold. Applying now Lemma \ref{L_boundedrep}
to $x$ and $y$, we infer that $x,y$ have representatives $\til{x}, \til{y}$ in $Z[X_1, \dots, X_r]$ with \eqref{Thue_boundsintheorem}.
\end{proof}

\subsection{Hyper- and superelliptic equations}

Recall that the polynomial
$$
F(X)=a_0X^n+a_1X^{n-1}+\dots+a_n \in A[X]
$$
has discriminant $D_F \ne 0$, that $\de \in A\setminus \{ 0 \}$,
and that for
$a_0, a_1, \dots , a_n, \de$ we have chosen representatives $\til{a_0}, \til{a_1}, \dots , \til{a_n}, \til{\de}\in \Z[X_1, \dots , X_r]$
satisfying \eqref{hypersuper_h_deg_assumption}.

Theorem \ref{T_hypersuper} will be deduced from the following Proposition, which has a meaning also if $q=0$.
Similarly as its analogue for Thue equations,
its proof is given in Sections \ref{S_boundthedegree}--\ref{S_boundtheheight}.

\begin{proposition}\label{P_hypersuper_over_B}
Let $w$ and $f$ be as in Propositions \ref{P_newdomaingenerators}, (i) and \ref{P_newdomaingenerators_spec_case}, respectively,
with the properties specified there, and consider the domain
$$
B:=A_0[f^{-1}, w].
$$
Further, let $m$ be an integer $\geq 2$, and assume that $n\geq 3$ if $m=2$
and $n\geq 2$ if $m\geq 3$.
Then for the solutions $x,y$ of the equation
\begin{equation} \label{hypersuper_over_B}
F(x)=\de y^m \quad\quad \text{in} \quad x,y\in B
\end{equation}
we have
\begin{align}
\label{hypersuper_degbound_B} &\odeg x,\ m\odeg y \leq (nd)^{\exp O(r)},
\\
\label{hypersuper_hbound_B} &
\oh(x),\ \oh(y) \leq \exp \left( m^3(nd)^{exp O(r)}(h+1) \right)
\end{align}
\end{proposition}

We now deduce Theorem \ref{T_hypersuper} from
Proposition \ref{P_hypersuper_over_B}.

\begin{proof}[Proof of Theorem \ref{T_hypersuper}]
Let $x,y$ be a solution of equation \eqref{hypersuper}. In view of \eqref{extended_domain_IV} $x,y$ is also a solution
in $B=A_0[f^{-1},w]$, where $f,w$ satisfy the conditions specified in Propositions \ref{P_newdomaingenerators}, (i) and \ref{P_newdomaingenerators_spec_case}, respectively. Then by
Proposition \ref{P_hypersuper_over_B}, \eqref{hypersuper_degbound_B} and \eqref{hypersuper_hbound_B} hold. Applying now Lemma \ref{L_boundedrep}
to $x$ and $y$, we infer that $x,y$ have representatives $\til{x}, \til{y}$ in $Z[X_1, \dots, X_r]$ with \eqref{hypersuper_boundsintheorem}.
\end{proof}

\begin{proposition}\label{P_ST_over_B}
Suppose that equation \eqref{hypersuper_over_B}
has a solution $x\in B$, $y\in B\cap\OQq$ and that also
$y\not= 0$ and $y$ is not a root of unity. Then
\begin{equation}\label{bound_ST_over_B}
m\leq \exp \left( (nd)^{\exp O(r)} (h+1) \right).
\end{equation}
\end{proposition}

\begin{proof}[Proof of Theorem \ref{T_ST_over_A}]
Let $x,y\in A$, $m\in\Zz_{\geq 2}$ be a solution of equation \eqref{hypersuper}.
First let $y\not\in\OQq$. Then $\odeg y\geq 1$, and together with
\eqref{hypersuper_degbound_B}
this implies \eqref{bound_ST_over_A-alg}.
Next, let $y\in\OQq$. Then Proposition \ref{P_ST_over_B} gives at once
\eqref{bound_ST_over_A}.
\end{proof}

The proof of Proposition \ref{P_ST_over_B} is a combination of results 
from Sections \ref{S_boundthedegree}--\ref{S_boundtheheight}.  
It is completed at the end of Section \ref{S_boundtheheight}.

\section{Bounding the degree}\label{S_boundthedegree}

In this section we shall prove \eqref{Thue_degbound_B} of Proposition \ref{P_Thue_over_B} and \eqref{hypersuper_degbound_B} of
Proposition \ref{P_hypersuper_over_B}.

We recall some results on function fields in one variable. Let $\K$ be an algebraically closed field of characteristic $0$, $z$ a
transcendental element over $\K$ and $M$ a finite extension of $\K(z)$. Denote by $g_{M/\K}$ the genus of $M$, and by $\M_M$ the
collection of valuations of $M/\K$, these are the discrete
valuations of $M$ with value group $\Z$ which are trivial on $\K$. Recall that these valuations satisfy the sum formula
$$
\sum_{v \in \M_M} v(\al) =0 \quad\quad \text{for} \quad \al\in M^*.
$$
For a finite subset $S$ of $\M_M$, an element $\al\in M$ is called
an $S$-integer if $v(\al) \geq 0$ for all $v\in \M_M\setminus S$.
The $S$-integers
form a ring in $M$, denoted by $\OO_S$. The (homogeneous) height of $\av=(\al_1, \dots, \al_l)\in M^l$ relative to $M/\K$ is defined by
$$
H_M(\av)=H_M(\al_1, \dots,\al_l):=-\sum_{v \in \M_M} \min (v(\al_1),\dots, v(\al_l)),
$$
and we define the height $H_M(f)$ of a polynomial $f\in M[X]$ by the height of the vector defined by the coefficients of $f$.
Further, we shall write $H_M(1, \av):=H_M(1,\al_1,\dots, \al_l)$. We note that
\begin{equation}\label{fnfield_heiht_proerties_I}
H_M(\al_i)\leq H_M(\av)\leq H_M(\al_1)+\dots+H_M(\al_l), \quad\quad i=1,\dots,l.
\end{equation}
By the sum formula,
\begin{equation}\label{fnfield_heiht_proerties_II}
H_M(\al \av)= H_M(\av) \quad\quad \text{for} \quad \al\in M^*.
\end{equation}
The height of $\al\in M$ relative to $M/\K$ is defined by
$$
H_M(\al):=H_M(1,\al)=-\sum _{v\in \M_M} \min(0,v(\al)).
$$
It is clear that $H_M(\al)=0$ if and only if $\al\in \K$. Using the sum formula, it is easy to prove that the height
has the properties
\begin{equation}\label{fnfield_heiht_proerties_III}
\begin{aligned}
&H_M(\al^l)= |l|H_M(\al),&&\\
&H_M(\al+\be)\leq H_M(\al)+H_M(\be), \quad\quad &&H_M(\al\be)\leq H_M(\al)+H_M(\be)
\end{aligned}
\end{equation}
for all non-zero $\al,\be \in M$ and for every integer $l$.

If $L$ is a finite extension of $M$, we have
\begin{equation}\label{fnfield_heiht_proerties_IV}
H_L(\al_0, \dots , \al_l)= [L:M] H_M(\al_0, \dots ,\al_l) \quad\quad \text{for} \quad \al_0, \dots, \al_l \in M.
\end{equation}
By $\deg f$ we denote the total degree of $f\in \K[z]$. Then for $f_0, \dots, f_l \in \K[z]$ with $\gcd(f_0, \dots , f_l)=1$ we have
\begin{equation}\label{fnfield_heiht_proerties_V}
H_{\K[z]}(f_0, \dots , f_l)=\max(\deg f_0, \dots , \deg f_l).
\end{equation}

\begin{lemma}\label{L_fnfield_I}
Let $\al_1, \dots , \al_l \in M$ and suppose that
$$
X^l+f_1X^{l-1}+\dots + f_l=(X-\al_1)\dots (X-\al_l)
$$
for certain $f_1, \dots f_l \in \K[z]$. Then
$$
[M:\K(z)]\max(\deg f_1, \dots , \deg f_l)=\sum_{i=1}^l H_M(\al_i).
$$
\end{lemma}

\begin{proof}
This is Lemma 4.1 in Evertse and Gy\H ory \cite{EGy9}.
\end{proof}

\begin{lemma}\label{L_fnfield_II}
Let
\[
F=f_0X^l+f_1X^{l-1}+\dots + f_l\in M[X]
\]
be a polynomial with $f_0\not= 0$ and with non-zero discriminant.
Let $L$ be the splitting field over $M$ of $F$.
Then
$$
g_{L/\K}\leq [L:M]\cdot \big(g_{M/\K}+l H_M(F)\big).
$$
In particular, if $M=\K (z)$ and $f_0\kdots f_l\in\K [z]$,
we have
\[
g_{L/\K}\leq [L:M]\cdot l\max (\deg f_0\kdots\deg f_l).
\]
\end{lemma}

\begin{proof}
The second assertion follows by combining the first assertion
with \eqref{fnfield_heiht_proerties_V}.
We now prove the first assertion. Our proof is a generalization
of that of Lemma H of Schmidt \cite{Sch7}.

For $v\in\M_M$, put $v(F):=\min (v(f_0)\kdots v(f_l))$.
Let $D_F$ denote the discriminant of $F$.
Since $D_F$ is a homogeneous polynomial of degree $2l-2$ in
$f_0\kdots f_l$, we have
\begin{equation}\label{extra-1}
v(D_F)\geq (2l-2)v(F).
\end{equation}
Let $S$ be the set of $v\in\M_M$ with $v(f_0)>v(F)$ or $v(D_F)>(2l-2)v(F)$.
We show that $L/M$ is unramified over every valuation $v\in \M_M\setminus S$.

Take $v\in\M_M\setminus S$. Let
\[
O_v:=\{ x\in M:\, v(x)\geq 0\},\ \ \ m_v:=\{ x\in M:\, v(x)>0\}
\]
denote the local ring at $v$, and the maximal ideal of $O_v$,
respectively. The residue class field $O_v/m_v$ is equal to $\K$
since $\K$ is algebraically closed.
Let $\varphi_v :\, O_v\to \K$
denote the canonical homomorphism.

Without loss of generality, we assume $v(F)=0$. Then $v(f_0)=0$, $v(D_F)=0$.
Let $\varphi_v(F):=\sum_{j=0}^l \varphi_v(f_j)X^{l-j}$.
Then $\varphi_v(f_0)\not= 0$ and $\varphi_v(F)$ has discriminant
$\varphi_v (D_F)\not= 0$. Since $D_F\not= 0$, the polynomial $F$
has $l$ distinct zeros in $L$, $\alpha_1\kdots\alpha_l$, say.
Further, $\varphi_v(F)$ has $l$ distinct zeros in $\K$,
$a_1\kdots a_l$, say.

Denote by $\Sigma_l$ the permutation group on $(1\kdots l)$.
Choose $c_1\kdots c_l\in \K$, such that the numbers
\[
\alpha_{\sigma} :=c_1\alpha_{\sigma (1)}+\cdots +c_l\alpha_{\sigma (l)}\ \
(\sigma\in\Sigma_l)
\]
are all distinct, and the numbers
\[
a_{\sigma}:=c_1a_{\sigma (1)}+\cdots +c_l a_{\sigma (l)}\ \ (\sigma\in\Sigma_l)
\]
are all distinct. Let $\alpha :=c_1\alpha_1+\cdots +c_l\alpha_l$.
Then $L=M(\alpha )$, and the monic minimal polynomial of $\alpha$ over $M$
divides $G:=\prod_{\sigma\in\Sigma_l} (X-\alpha_{\sigma})$
which by the theorem of symmetric functions belongs to $M[X]$.
The image of $G$ under $\varphi_v$ is $\prod_{\sigma\in\Sigma_l} (X-a_{\sigma})$
and this has only simple zeros.
This implies that $L/M$ is unramified at $v$.

For $v\in\M_M$ and any valuation $\in\M_L$ above $v$,
denote by $e(V|v)$ the ramification index of $V$ over $v$.
Recall that $\sum_{V|v} e(V|v)=[L:M]$, where the sum is taken over
all valuations of $L$ lying above $v$.
Now the Riemann-Hurwitz formula implies that
\begin{eqnarray}\label{extra-2}
2g_{L/\K}-2&=& [L:M](2g_K-2)+\sum_{v\in S}\sum_{V|v}(e(V|v)-1)
\\
\nonumber
&\leq& [L:M](2g_K-2+|S|),
\end{eqnarray}
where $|S|$ denotes the cardinality of $S$.
It remains to estimate $|S|$. By the sum formula and \eqref{extra-1}
we have
\begin{eqnarray*}
|S|&\leq&\sum_{v\in S}\Big( (v(f_0)-v(F))+(v(D_F)-(2l-2)v(F))\Big)
\\
&=& -\sum_{v\in S} (2l-1)v(F)-\sum_{v\in\M_M\setminus S} v(f_0)-\sum_{v\in\M_M\setminus S} v(D_F)
\\
&\leq& -(2l-1)\sum_{v\in\M_M} v(F)=(2l-1)H_M(F).
\end{eqnarray*}
By inserting this into \eqref{extra-2} we arrive at an inequality
which is stronger than what we wanted to prove.
\end{proof}

In the sequel we keep the notation of Proposition \ref{P_newdomaingenerators}.
To prove \eqref{Thue_degbound_B} and \eqref{hypersuper_degbound_B} we may suppose that $q>0$ since the case $q=0$ is trivial.
Let again $K_0:=\Q(z_1, \dots , z_q)$, $K:=K_0(w)$, $A_0:=\Z[z_1, \dots , z_q]$, $B:=\Z[z_1, \dots , z_q, f^{-1},w]$ with $f,w$ specified
in Propositions \ref{P_newdomaingenerators} (i) and
\ref{P_newdomaingenerators_spec_case}.

Fix $i\in \{ 1, \dots , q \}$. Let $\K_i:=\Q(z_1, \dots, z_{i-1},z_{i+1}, \dots ,z_q)$ and $\obK_i$ its algebraic closure. Then
$A_0$ is contained in $\obK_i[z_i]$. Denote by $w^{(1)}:=w, \dots, w^{(D)}$ the conjugates of $w$ over $K_0$. Let
$M_i$ denote the splitting field of the polynomial $X^D+\F_1X^{D-1}+\dots +\F_D$ over $\obK_i(z_i)$, that is
$$
M_i:=\obK_i(z_i,w^{(1)}, \dots, w^{(D)}).
$$
Then
$$
B_i:=\obK_i[z_i,f^{-1},w^{(1)}, \dots, w^{(D)}]
$$
is a subring of $M_i$ which contains $B=\Z[z_1, \dots, z_q, f^{-1},w]$ as a subring. Let $\Delta_i:=[M_i:\obK_i(z_i)]$. Further,
let $g_{M_i}$ denote the genus of $M_i/\obK_i$, and $H_{M_i}$ the height taken with respect to $M_i/\obK_i$. Put
\begin{equation}\label{define_d1h1}
d_1:=\max(d_0, \deg f, \deg \F_1, \dots, \deg \F_D).
\end{equation}
We mention that in view of Propositions \ref{P_newdomaingenerators},
\ref{P_newdomaingenerators_spec_case},
\begin{equation}\label{bound_d1_by d}
d_1 \leq (nd)^{\exp O(r)}.
\end{equation}

\begin{lemma}\label{L_bound_degbar_by_sumofheightsofconjugates}
Let $\al\in K^*$ and denote by $\al^{(1)}, \dots, \al^{(D)}$ the conjugates of $\al$ corresponding to $w^{(1)}, \dots, w^{(D)}$. Then
$$
\odeg \al \leq qDd_1+\sum_{i=1}^q \Delta_i^{-1} \sum_{j=1}^D H_{M_i}(\al^{(j)}).
$$
\end{lemma}

\begin{proof}
This is Lemma 4.4 in Evertse and Gy\H ory \cite{EGy9}.
\end{proof}

Conversely, we have the following:

\begin{lemma}\label{L_bound_heightsofconjugates_by_degbar}
Let $\al\in K^*$ and $\al^{(1)}, \dots, \al^{(D)}$ be as in Lemma \ref{L_bound_degbar_by_sumofheightsofconjugates}. Then we have
\begin{equation}\label{bound_heightsofconjugates_by_degbar}
\max_{i,j} H_{M_i}(\al^{(j)}) \leq \Delta_i\left(2D \odeg \al + (2d_0)^{\exp O(r)}\right).
\end{equation}
\end{lemma}

\begin{proof}
Consider the representation of the form \eqref{representation_typeII} of $\al$.
Since $P_{\al,k},Q\in K_0$, we have
$$
\al^{(j)}=\sum_{k=0}^{D-1} \frac{P_{\al,k}}{Q}\left( w^{(j)}\right)^k \quad\quad \text{for} \quad j=1,\dots, D.
$$
In view of \eqref{fnfield_heiht_proerties_III} it follows that
\begin{equation}\label{represent_alj}
H_{M_i}(\al^{(j)})\leq \sum_{k=0}^{D-1} H_{M_i}\left( \frac{P_{\al,k}}{Q}\right) + \sum_{k=0}^{D-1}  kH_{M_i}\left( w^{(j)}\right).
\end{equation}
But we have
\begin{equation}\label{bound_height_intermediate_I}
\begin{aligned}
H_{M_i}\left( \frac{P_{\al,k}}{Q}\right) &\leq \Delta_i H_{\K_i(z)}\left( \frac{P_{\al,k}}{Q}\right) \leq \Delta_i (\deg_{z_i}P_{\al,k}+\deg_{z_i} Q )\\
& \leq \Delta_i (\deg P_{\al,k}+\deg Q ) \leq 2 \Delta_i \odeg \al.
\end{aligned}
\end{equation}
Further, applying Lemma \ref{L_fnfield_I} with $M_i, w^{(1)},\dots , w^{(D)}$ instead of $M,\al_1,\dots, \al_l$, we get
\begin{equation}
\begin{aligned}\label{bound_height_intermediate_II}
H_{M_i}\left(w^{(j)} \right) &\leq \Delta_i \max_{1\leq j \leq D} (\deg_{z_i} \F_j)\\
& \leq \Delta_i \max_{1\leq j \leq D} (\deg \F_j) \leq \Delta_i (2d_0)^{\exp O(r)}.
\end{aligned}
\end{equation}
Now using the fact that $D\leq d_0^{\rho} \leq d_0^{r-1}$, \eqref{represent_alj}, \eqref{bound_height_intermediate_I} and
\eqref{bound_height_intermediate_II} imply \eqref{bound_heightsofconjugates_by_degbar}.
\end{proof}

\subsection{Thue equations}

As before, $\K$ is an algebraically closed field of characteristic $0$, $z$ a
transcendental element over $\K$ and $M$ a finite extension of $\K(z)$. Further, $g_{M/\K}$ denotes the genus of $M$, $\M_M$ the
collection of valuations of $M/\K$, and for a finite subset $S$ of $\M_M$, $\OO_S$ denotes the ring of $S$-integers in $M$.
We denote by $|S|$ the cardinality of $S$.

Consider now the Thue equation
\begin{equation}\label{Thue_over_fnfield}
F(x,y)=1 \quad\quad \text{in} \quad x,y\in\OO_S,
\end{equation}
where $F$ is a binary form of degree $n\geq 3$ with coefficients in $M$
and with non-zero discriminant.

\begin{proposition}\label{P_Thue_over_fnfields}
Every solution $x,y\in \OO_S$ of \eqref{Thue_over_fnfield} satisfies
\begin{equation}\label{Thue_bound_fnfield}
\max(H_M(x),H_M(y)) \leq 89 H_M(F)+212 g_{M/\K}+|S|-1.
\end{equation}
\end{proposition}

\begin{proof}
This is Theorem 1, (ii) of Schmidt \cite{Sch7}.
\end{proof}

We note that from Mason's fundamental inequality concerning $S$-unit equations over function fields (see Mason \cite{Mason1})
one could deduce \eqref{Thue_bound_fnfield} with smaller constants than $89$ and $212$. However, this is irrelevant
for the bounds in \eqref{Thue_boundsintheorem}.

Now we use Proposition \ref{P_Thue_over_fnfields} to prove the statement \eqref{Thue_degbound_B} of Proposition \ref{P_Thue_over_B}.

\begin{proof}[Proof of \eqref{Thue_degbound_B}]
We denote by $w^{(1)}:=w, \dots, w^{(D)}$ the conjugates of $w$ over $K_0$, and for
$\al\in K$ we denote by $\al^{(1)}, \dots, \al^{(D)}$ the conjugates of $\al$ corresponding to $w^{(1)}, \dots, w^{(D)}$.

Next, for $i=1\kdots n$
we put $\K_i:=\Q(z_1, \dots, z_{i-1},z_{i+1}, \dots ,z_q)$ and
denote by $\obK_i$ its algebraic closure.
Further, $M_i$ denotes the splitting field of the polynomial
$X^D+\F_1X^{D-1}+\dots +\F_D$ over $\obK_i(z_i)$, we put
$\Delta_i:=[M_i:\obK_i(z_i)]$ and define
$$
S_i:=\{ v \in \M_{M_i} \ :\ v(z_i)<0 \ \text{or} \  v(f)>0 \}.
$$
The conjugates $w^{(j)}$ ($j=1\kdots D$) lie in $M_i$ and are
all integral over $\K_i[z_i]$. Hence they belong to $\OO_{S_i}$.
Further, $f^{-1}\in \OO_{S_i}$.
Consequently, if $\alpha\in B=A_0[f^{-1},w]$, then
$\alpha^{(j)}\in \OO_{S_i}$ for $j=1\kdots D$, $i=1\kdots q$.

Let $x,y$ be a solution of equation \eqref{Thue_over_B}.
Put $F':=\delta^{-1}F$, and let $F'^{(j)}$ be the binary form obtained
by taking the $j$-th conjugates of the coefficients of $F'$.
Let $j\in\{ 1\kdots D\}$, $i\in\{ 1\kdots q\}$.
Then clearly, $F'^{(j)}\in M_i[X,Y]$, and
\[
F'^{(j)}(x^{(j)}, y^{(j)})=1, \quad\quad x^{(j)}, y^{(j)}\in \OO_{S_i}.
\]
So by Proposition \ref{P_Thue_over_fnfields} we obtain that
\begin{equation}\label{bound_fnfield_xjyj}
\max(H_{M_i}(x^{(j)}), H_{M_i}(y^{(j)})) \leq
89H_{M_i}(F^{(j)})+212 g_{M_i}+|S_i|-1.
\end{equation}
We estimate the various parameters in this bound.
We start with $H_{M_i}(F'^{(j)})$. We recall that
$F'(X,Y)=\delta^{-1}(a_0X^n+a_1X^{n-1}Y+\dots +a_nY^n)$.
Using \eqref{fnfield_heiht_proerties_II}, \eqref{fnfield_heiht_proerties_I}
and Lemma \ref{L_bound_heightsofconjugates_by_degbar} we infer that
\begin{eqnarray*}
H_{M_i}(F'^{(j)})&=& H_{M_i}(a_0^{(j)}\kdots a_n^{(j)}) \leq H_{M_i}(a_0^{(j)})+\dots+H_{M_i}(a_n^{(j)})
\\
&\leq& \Delta_i \left(2D(\odeg a_0+\dots +\odeg a_n)+n(2d_0)^{\exp O(r)}\right).
\end{eqnarray*}
By Lemma \ref{L_degbar_hbar_by_d*_h*} we have
$$
\odeg a_i \leq (2d^*)^{\exp O(r)} \quad\quad \text{for} \quad i=0,\dots,n,
$$
where $d^*:=\max(d_0,\deg \til{a}_i) \leq d$.
Further, we have $d_0\leq d$, $D\leq d_0^{r-q}\leq d^r$.
Thus we obtain that
\begin{eqnarray}\label{bound_height_of_Fj}
H_{M_i}(F'^{(j)}) &\leq& \Delta_i\big( 2D(n+1)(2d)^{\exp O(r)}+n(2d)^{\exp O(r)}\big)
\\
\nonumber
&\leq& \Delta_i(nd)^{\exp O(r)}.
\end{eqnarray}
Next, we estimate the genus. Using Lemma \ref{L_fnfield_II} with
$F(X)=\F(X)=X^D+\F_1X^{D-1}+\cdots +\F_D$,
applying Proposition \ref{P_newdomaingenerators}, and using
$d_0\leq d$, $D\leq d_0^r \leq d^r$,
we infer that
\begin{equation}\label{bound_gMi}
g_{M_i}\leq \Delta_i D \max_{1 \leq k \leq D} \deg_{z_i} \F_k \leq \Delta_i D (2d_0)^{\exp O(r)}\leq \Delta_i (nd)^{\exp O(r)}.
\end{equation}
Lastly, we estimate $|S_i|$.
Each valuation of $\overline{\K}_i(z_i)$
can be extended to at most $[M_i:\overline{\K}_i(z_i)]=\Delta_i$ valuations of
$M_i$. Thus $M_i$ has at most
$\Delta_i$ valuations $v$ with $v(z_i)<0$ and at most $\Delta_i \deg f$ valuations $v$ with $v(f)>0$.
Hence using Proposition \ref{P_newdomaingenerators_spec_case}, we get
\begin{equation}\label{bound_Si}
|S_i| \leq \Delta_i+\Delta_i \deg_{z_i} f \leq \Delta_i(1+\deg f)
\leq \Delta_i (nd)^{\exp O(r)}.
\end{equation}

By inserting the bounds \eqref{bound_height_of_Fj}, \eqref{bound_gMi} and
\eqref{bound_Si} into \eqref{bound_fnfield_xjyj},
we infer
\begin{equation}\label{bound_fnfield_xjyj_II}
\max(H_{M_i}(x^{(j)}), H_{M_i}(y^{(j)})) \leq \Delta_i (nd)^{\exp O(r)}.
\end{equation}
In view of Lemma \ref{L_bound_degbar_by_sumofheightsofconjugates}, \eqref{bound_fnfield_xjyj_II}, $D\leq d^r$, $q\leq r$ and \eqref{bound_d1_by d}
we deduce that
$$
\odeg x, \odeg y \leq qDd_1+\sum_{i=1}^q \Delta_i^{-1}\sum_{j=1}^D H_{M_i}(x^{(j)}) \leq (nd)^{\exp O(r)}.
$$
This proves \eqref{Thue_degbound_B}.
\end{proof}

\subsection{Hyper- and superelliptic equations}

Recall the notation introduced at the beginning of Section \ref{S_boundthedegree}.
Again, $\K$ is an algebraically closed field of characteristic $0$, $z$ a
transcendental element over $\K$, $M$ a finite extension of $\K(z)$, and
$S$ a finite subset of $\M_M$.

\begin{proposition}\label{P_superelliptic_over_fnfileds}
Let $F \in M[X]$ be a polynomial with non-zero discriminant and $m\geq 3$ a given integer.
Put $n:=\deg F$ and assume $n\geq 2$. All solutions of the equation
\begin{equation}\label{super-fnfield}
F(x)=y^m \quad\quad \text{in} \quad x,y \in \OO_S
\end{equation}
have the property
\begin{eqnarray}
\label{superbound-x}
H_M(x)&\leq& (6n+18)H_M(F)+6g_{M/\K}+2|S|,
\\
\label{superbound-y}
mH_M(y)&\leq& (6n^2+18n+1)H_M(F)+6ng_{M/\K}+2n|S|.
\end{eqnarray}
\end{proposition}

\begin{proof}
First assume that $F$ splits into linear factors over $M$,
and that $S$ consists only of the infinite valuations of $M$,
these are the valuations of $M$ with $v(z)<0$. Under these hypotheses,
Mason \cite[p.118, Theorem 15]{Mason1}, proved that
for every solution $x,y$ of \eqref{super-fnfield} we have
\begin{equation}\label{Mason-superbound}
H_M(x)\leq 18H_M(F)+6g_{M/\K}+2(|S|-1).
\end{equation}
But Mason's proof remains valid without any changes
for any arbitrary finite set of places $S$.
That is, \eqref{Mason-superbound} holds if $F$ splits into linear factors over $M$,
without any condition on $S$.

We reduce the general case, where  the splitting field of $M$ may be larger
than $M$, to the case considered by Mason.
Let $L$ be the splitting field of $F$ over $M$,
and $T$ the set of valuations of $L$ that extend those of $S$.
Then $|T|\leq [L:M]\cdot |S|$, and by Lemma \ref{L_fnfield_II},
we have $g_{L/\K}\leq [L:M]\cdot (g_{M/\K}+nH_M(F))$.
Note that \eqref{Mason-superbound} holds, but with $L,T$
instead of $M,S$. It follows that
\begin{eqnarray*}
[L:M]\cdot H_M(x)=H_L(x)&\leq& 18H_L(F)+6g_{L/\K}+2(|T|-1)
\\
&\leq& [L:M]\big( (6n+18)H_M(F)+6g_{M/\K}+2|S|\big)
\end{eqnarray*}
which implies \eqref{superbound-x}.
Further,
\begin{equation}\label{x-to-y}
mH_M(y)=H_M(y^m)=H_M(F(x))\leq H_M(F)+nH_M(x),
\end{equation}
which gives \eqref{superbound-y}.
\end{proof}

\begin{proposition}\label{P_hyperelliptic_over_fnfileds}
Let $F \in M[X]$ be a polynomial with non-zero discriminant. Put $n:=\deg F$ and assume $n\geq 3$. Then the solutions of
\begin{equation}\label{hyper-fnfield}
F(x)=y^2 \quad\quad \text{in} \quad x,y \in \OO_S
\end{equation}
have the property
\begin{eqnarray}
\label{hyperbound-x}
H_M(x)&\leq& (42n+37)H_M(F)+8g_{M/\K}+4|S|,
\\
\label{hyperbound-y}
H_M(y)&\leq& (21n^2+19n)H_M(F)+4ng_{M/\K}+2n|S|.
\end{eqnarray}
\end{proposition}

\begin{proof}
First assume that $F$ splits into linear factors over $M$,
that $S$ consists only of the infinite valuations of $M$,
that $F$ is monic, and that $F$ has its coefficients in $\OO_S$.
Under these hypotheses, Mason \cite[p.30, Theorem 6]{Mason1}
proved that for every solution of \eqref{hyper-fnfield} we have
\begin{equation}\label{mason-hyperbound}
H_M(x)\leq 26H_M(F)+8g_{M/\K}+4(|S|-1).
\end{equation}
An inspection of Mason's proof shows that his result is valid
for arbitrary finite sets of valuations $S$, not just the set of infinite
valuations. This leaves only the conditions imposed on $F$.

We reduce the general case to the special case to which
\eqref{mason-hyperbound} is applicable.
Let $F=a_0X^n+\cdots +a_n$.
Let $L$ be the splitting field of $F\cdot (X^2-a_0)$ over $M$.
Let $T$ be the set of valuations of $L$ that extend the valuations
of $S$, and also the valuations $v\in \M_M$ such that $v(F)<0$.
Further, let $F'=X^n+a_1X^{n-1}+a_0a_1X^{n-2}+\cdots +a_0^{n-1}a_n$,
and let $b$ be such that $b^2=a_0^{n-1}$.
Then for every solution $x,y$ of \eqref{hyper-fnfield} we have
\[
F'(a_0x)=(by)^2,\ \ a_0x,by\in\OO_T,
\]
and moreover, $F'\in\OO_T[X]$, $F'$ is monic, and $F'$ splits into linear
factors over $L$. So by \eqref{mason-hyperbound},
\begin{equation}\label{La0-bound}
H_L(a_0x)\leq 26H_L(F')+8g_{L/\K}+4(|T|-1).
\end{equation}
First notice that
\[
H_L(F')=[L:M]H_M(F')\leq [L:M]\cdot nH_M(F).
\]
Further,
\[
|T|\leq [L:M]\Big(|S|-\sum_{v\in\M_M}\min (0,v(F))\Big)\leq [L:M]\big(|S|+H_M(F)\big).
\]
Finally, by $H_M(F\cdot (X^2-a_0))\leq 2H_M(F)$ and Lemma \ref{L_fnfield_II},
we have
\[
g_{L/\K}\leq [L:M](g_{M/\K}+(n+2)2H_M(F)).
\]
By inserting these bounds into \eqref{La0-bound}, we infer
\begin{eqnarray*}
[L:M]H_M(x)&\leq& [L:M]\big(H_M(a_0x)+H_M(F)\big)=H_L(a_0x)+[L:M]H_M(F)
\\
&\leq& [L:M]\big( (42n+37)H_M(F)+8g_{M/\K}+4|S|\big).
\end{eqnarray*}
This implies \eqref{hyperbound-x}.
The other inequality \eqref{hyperbound-y} follows by combining
\eqref{hyperbound-x} with \eqref{x-to-y} with $m=2$.
\end{proof}

The final step of this subsection is to prove statement \eqref{hypersuper_degbound_B} in Proposition \ref{P_hypersuper_over_B}.

\begin{proof}[Proof of \eqref{hypersuper_degbound_B}]
We closely follow the proof of statement \eqref{Thue_degbound_B}
in Proposition \ref{P_Thue_over_B}, and use the same notation.
In particular, $\K_i,M_i,S_i,\Delta_i$ will have the same meaning,
and for $\alpha\in B$, $j=1\kdots D$, the $j$-th conjugate $\alpha^{(j)}$
is the one corresponding to $w^{(j)}$.
Put $F':=\delta^{-1}F$, and let $F'^{(j)}$ be the polynomial obtained
by taking the $j$-th conjugates of the coefficients of $F'$.

We keep the argument together for both hyper- and superelliptic equations by using the worse bounds everywhere.
Let $x,y\in B$ be a solution of \eqref{hypersuper},
where $m,n\geq 2$ and $n\geq 3$ if $m=2$. Then
\[
F'^{(j)}(x^{(j)})=(y^{(j)})^m,\ \ x^{(j)},y^{(j)}\in\OO_{S_i}.
\]
By combining Propositions \ref{P_superelliptic_over_fnfileds}
and \ref{P_hyperelliptic_over_fnfileds} we obtain the generous bound
\[
H_{M_i}(x^{(j)}),\ mH_{M_i}(y^{(j)})\ \leq 80n^2\big(H_{M_i}(F'^{(j)})+g_{M_i/\K_i}+|S_i|\big).
\]
For $H_{M_i}(F'^{(j)})$, $g_{M_i/\K_i}$, $|S_i|$ we have precisely the same
estimates as \eqref{bound_height_of_Fj}, \eqref{bound_gMi},
\eqref{bound_Si}. Then a similar computation as in the proof of
\eqref{Thue_degbound_B} leads to
\begin{equation}\label{Mibound_for_xy}
H_{M_i}(x^{(j)}),\ mH_{M_i}(y^{(j)})\ \leq \Delta_i (nd)^{\exp O(r)}.
\end{equation}

Now employing Lemma \ref{L_bound_degbar_by_sumofheightsofconjugates}
and ignoring for the moment $m$
we get similarly as in the proof of \eqref{Thue_degbound_B},
\[
\odeg x,\ \odeg y \leq (nd)^{\exp O(r)}.
\]

It remains to estimate $m\odeg y$.
If $y\in\OQq$ we have $\odeg y =0$.
Assume that $y\not\in\OQq$. Then $y\not\in\K_i$ for at least one index $i$.
Since $y\in B\subset \K_i(z_i,w)$ and $[\K_i(z_i,w):\K_i(z_i)]\leq D$, we have
\[
H_{M_i}(y)=[M_i:\K_i(z_i,w)]H_{\K_i(z_i,w)}(y)\geq [M_i:\K_i(z_i,w)]\geq\Delta_i/D.
\]
Together with \eqref{Mibound_for_xy} and $D\leq d^r$ this implies
\[
m\leq (nd)^{\exp O(r)}.
\]
This concludes the proof of \eqref{hypersuper_degbound_B}.
\end{proof}

\section{Specializations}\label{S_specializations}

In this section we shall consider specialization homomorphisms from the domain $B$ to $\oQ$, and using these specializations together with earlier results
concerning our equations in the number field case
we shall finish the proof of Propositions \ref{P_Thue_over_B} and \ref{P_hypersuper_over_B}.

We start with some notation. The set of places of $\Q$ is $\M_{\Q}=\{\infty\} \cup \{\mbox{primes}\}$. By $|\cdot|_{\infty}$ we denote
the ordinary absolute value on $\Q$ and by $|\cdot|_p$ ($p$ prime) the $p$-adic absolute value with $|p|_p=p^{-1}$. More generally, let
$L$ be an algebraic number field with set of places $\M_L$. Given $v\in \M_L$, we define the absolute value $|\cdot|_v$ in such a
way that its restriction to $\Q$ is $|\cdot|_p$ if $v$ lies above $p\in \M_{\Q}$.
These absolute values satisfy the product formula
$$
\prod_{v\in \M_L} |\al|_v^{d_v}=1 \quad\quad \text{for} \quad \al\in L^*,
$$
where $d_v:=[L_v:\Qq_p]/[L:\Qq ]$, with $p\in\M_{\Qq}$ the place below $v$,
and $\Qq_p$, $L_v$ the completions of $\Qq$ at $p$, $L$ at $v$.
Note that we have $\sum_{v \mid p} d_v =1$ for every $p\in\M_{\Qq}$.
The absolute logarithmic height of $\al \in L$ is defined by
$$
h(\al):=\log \prod_{v\in \M_L} \max(1,|\al|_v^{d_v}).
$$
This depends only on $\al$ and not on the choice of the number field $L$ containing $\al$, hence it defines a height on $\oQ$.
For properties of the height we refer to Bombieri and Gubler \cite{BombGub}.

\begin{lemma}\label{L_bound_hbar_for_q_eq_0}
Let $m \geq 1$ and let $\al_1, \dots ,\al_m \in \oQ$ be distinct, and suppose that $G(X):=\prod_{j=1}^m (X-\al_j)\in \Z[X]$.
Let $q,p_0, \dots ,p_{m-1}$ be integers with $\gcd(q,p_0, \dots ,p_{m-1})=1$ and put
$$
\be_j:=\sum_{i=0}^{m-1} \frac{p_j}{q} \al_j^i, \quad\quad j=1, \dots ,m.
$$
Then
$$
\log\max(|q|,|p_0|, \dots ,|p_{m-1}|) \leq 2m^2+(m-1)h(G)+\sum_{j=1}^m h(\be_j).
$$
\end{lemma}

\begin{proof}
This is Lemma 5.2 in Evertse and Gy\H ory \cite{EGy9}.
\end{proof}

We now consider our specializations $B \mapsto \oQ$ and prove some of their properties. These specializations were
introduced by Gy\H ory \cite{Gy3} and \cite{Gy20} and, in a refined form, by Evertse and Gy\H ory \cite{EGy9}.

We assume $q>0$ and apart from that keep the notation and assumption from Section \ref{S_reduction}. In particular,
$K_0:=\Q(z_1, \dots , z_q)$, $K:=\Q(z_1, \dots , z_q,w)$, $A_0:=\Z[z_1, \dots , z_q]$. Further, $B:=\Z[z_1, \dots , z_q, f^{-1},w]$
where $f$ is a non-zero element of $A_0$ with the properties specified in Proposition \ref{P_newdomaingenerators_spec_case}, and
$w$ is integral over $A_0$ and has minimal polynomial
$$
\F(X)=X^D+\F_1 X^{D-1}+\dots +\F_D \in A_0[X]
$$
over $K_0$ as in Proposition \ref{P_newdomaingenerators} (i). In the case $D=1$ we take $w=1$, $\F(X)=X-1$.

Let $\uv=(u_1, \dots, u_q) \in \Z^q$. Then the substitution $z_1 \to u_1, \dots, z_q \to u_q$ defines a ring homomorphism
(specialization) from $K_0$ to $\Q$
$$
\vp_{\uv}: \al \mapsto \al(\uv): \left\{ \al=\frac{g_1}{g_2} : g_1, g_2 \in A_0, g_2(\uv) \ne 0 \right\} \to \Q.
$$
To extend this to a ring homomorphism from $B$ to $\oQ$ we have to impose some restrictions on $\uv$. Let $\Delta_{\F}$ be the
discriminant of $\F$ (with $\Delta_{\F}=1$ if $D=1$), and let
\begin{equation}\label{define_H}
\HH:=\Delta_{\F} \cdot \F_D \cdot f.
\end{equation}
Put
\begin{equation}\label{define d0*_d1*_h0*_h1*}
\begin{cases}
d_0^*:=\max(\deg \F_1, \dots, \deg \F_D), \quad\quad d_1^*:=\max(d_0^*, \deg f) \\
h_0^*:=\max(h(\F_1), \dots, h(\F_D)), \quad\quad h_1^*:=\max(h_0^*, h(f)).
\end{cases}
\end{equation}
Clearly $\HH \in A_0$ and since $\Delta_{\F}$ is a homogeneous polynomial in $\F_1, \dots , \F_D$ of
degree $2D-2$, we have
\begin{equation}\label{bound_deg_H}
\deg \HH \leq (2D-1)d_0^*+d_1^*.
\end{equation}
Further, by Proposition \ref{P_newdomaingenerators} (i), Proposition \ref{P_newdomaingenerators_spec_case} and \eqref{Thue_h_deg_assumption}
we also have
\begin{equation}\label{bound_d0*_h0*}
\left\{
\begin{aligned}
&d_0^*\leq (2d)^{\exp O(r)}, \quad\quad h_0^*\leq (2d)^{\exp O(r)}(h+1), \\
&d_1^*\leq (nd)^{\exp O(r)}, \quad\quad h_1^*\leq (nd)^{\exp O(r)}(h+1)
\end{aligned}
\right.
\end{equation}

Next assume that
\begin{equation}\label{condition_on_u}
\HH(\uv)\ne 0.
\end{equation}
Then we have $f(\uv)\ne 0$, $\Delta_F(\uv)\ne 0$, hence the polynomial
$$
\F_{\uv}:=X^D+\F_1(\uv)X^{D-1}+\dots +\F_D(\uv)
$$
has $D$ distinct zeros which are all different from $0$, say $w^{(1)}(\uv), \dots, w^{(D)}(\uv)$. Consequently, for $j=1, \dots , D$ the assignment
$$
z_1 \mapsto u_1, \dots, z_q \mapsto u_q, w \mapsto w^{(j)}(\uv)
$$
defines a ring homomorphism $\vp_{\uv,j}$ from $B$ to $\oQ$; if $D=1$ it is just $\vp_{\uv}$. The image of $\al\in B$ under
$\vp_{\uv,j}$ is denoted by $\al^{(j)}(\uv)$. It is important to note that if $\al$ is a unit in $B$, then its image by a specialization
cannot be $0$. Thus by Proposition \ref{P_newdomaingenerators_spec_case}, $\de(\uv)\ne 0$ and $D_F(\uv)\ne 0$.

Recall that we may express elements of $B$ as
\begin{eqnarray}
&&\al=\sum_{i=1}^{D-1}\left( P_i/Q \right) w^i
\\
\nonumber
&&\quad\quad \mbox{where }P_0, \dots, P_{D-1}, Q \in A_0, \ \gcd(P_0, \dots, P_{D-1}, Q)=1.
\end{eqnarray}
Because of $\al\in B$, $Q$ must divide a power of $f$; hence $Q(\uv)\ne 0$. So we have
\begin{equation}
\al^{(j)}(\uv)=\sum_{i=1}^{D-1}\left( P_i(\uv)/Q(\uv) \right) \left(w^{(j)}(\uv)\right)^i, \quad\quad j=1, \dots, D.
\end{equation}
Clearly, $\vp_{\uv,j}$ is the identity on $B\cap \Qq$.
Hence if $\al\in B\cap \oQ$ then $\vp_{\uv,j}(\al)$ has the same minimal
polynomial as $\al$ and so it is a conjugate of $\al$.

For $\uv=(u_1, \dots , u_q)\in \Z^q$, put $|\uv|:=\max(|u_1|, \dots , |u_q|)$. It is easy to check that for any $g\in A_0$, $\uv\in \Z^q$
\begin{equation}
\log |g(\uv)| \leq q\log \deg g + h(g)+ \deg g \log \max (1,|\uv|).
\end{equation}
In particular, we have
\begin{equation}
h(\F_{\uv}) \leq q\log d_0^* + h_0^* + d_0^* \log \max (1,|\uv|)
\end{equation}
and so by Lemma 5.1 of Evertse and Gy\H ory \cite{EGy9}
\begin{equation}
\sum_{j=1}^D h(w^{(j)}(\uv)) \leq D+1+q\log d_0^* + h_0^* + d_0^* \log \max (1,|\uv|).
\end{equation}

We define the algebraic number fields $K_{\uv, j}=\Q(w^{(j)}(\uv))$ for $j=1, \dots , D$. We denote by $\Delta_L$ the
the discriminant of an algebraic number field $L$. We derive an upper bound for the absolute value of the discriminant
$\Delta_{K_{\uv, j}}$ of $K_{\uv, j}$.

\begin{lemma}\label{L_bound_fielddisc_spec}
Let $\uv\in \Z^q$ with $\HH(\uv)\ne 0$. Then for $j=1,\dots, D$ we have $[K_{\uv, j}:\Q]\leq D$ and
$$
|\Delta_{K_{\uv, j}}|\leq D^{2D-1}\left( (d_0^*)^q e^{h_0^*} \max (1,|\uv|^{d_0^*})\right)^{2D-2}.
$$
\end{lemma}

\begin{proof}
This is Lemma 5.5 in Evertse and Gy\H ory \cite{EGy9}.
\end{proof}

The following two lemmas relate the height of $\al\in B$ to the heights of $\al^{(j)}(\uv)$ for $\uv\in \Z^q$.

\begin{lemma}\label{L_bound_ha_by_degbar_hbar}
Let $\uv\in\Z^q$ with $\HH(\uv)\ne 0$, and let $\al\in B$. Then for $j=1,\dots, D$,
\begin{eqnarray*}
&&h(\al^{(j)}(\uv)) \leq D^2+q(D\log d_0^*+\log\odeg \al)+
\\
&&\qquad\qquad\qquad
+Dh_0^*+\oh(\al)+(D d_0^*+\odeg \al)\log\max(1,|\uv|).
\end{eqnarray*}
\end{lemma}

\begin{proof}
This is Lemma 5.6 in Evertse and Gy\H ory \cite{EGy9}.
\end{proof}

\begin{lemma}\label{L_bound_hbar by_ha_of conjugates of specializations}
Let $\al\in B$, $\al\ne 0$, and let $N$ be an integer with
\begin{equation}\label{define_Nal}
N \geq \max(\odeg \al, 2Dd_0^*+2(q+1)(d_1^*+1)).
\end{equation}
Then the set
$$
\s:=\left\{ \uv\in \Z^q \ : \ |\uv|\leq N, \HH(\uv)\ne 0  \right\}
$$
is non-empty, and
\begin{equation}\label{bound_hbar by_ha_of conjugates of specializations}
\oh(\al)\leq 5 N^4(h_1^*+1)^2+2D(h_1^*+1)H,
\end{equation}
where $H:=\max\{ h(\al^{(j)}(\uv)) \ : \ \uv \in \s, \ j=1, \dots, D \}$.
\end{lemma}

\begin{proof}
This is Lemma 5.7 in Evertse and Gy\H ory \cite{EGy9}.
\end{proof}

\section{Bounding the height and the exponent $m$}\label{S_boundtheheight}

We shall derive the height bounds
\eqref{Thue_hbound_B} in Proposition \ref{P_Thue_over_B}
and \eqref{hypersuper_hbound_B} in Proposition \ref{P_hypersuper_over_B},
as well as the upper bound for $m$ in
Proposition \ref{P_ST_over_B} by combining the specialization techniques
from the previous section with existing effective results for Diophantine
equations over $S$-integers of a number field,
namely Gy\H{o}ry and Yu \cite{GyYu1}
for Thue equations,
and the three authors \cite{BA18} for hyper- and superelliptic equations
and the Schinzel-Tijdeman equation.

\subsection{Thue equations}

In the statement of the result of Gy\H{o}ry and Yu we need some notation.

For an algebraic number field $L$, we denote by
$d_L$, $\OO_L$, $\M_L$, $\Delta_L$, $h_L$, $r_L$ and $R_L$ the degree, ring of integers,
set of places, discriminant, class number, unit rank and regulator of $L$.
The absolute norm of an ideal $\fa$ of $\OO_L$ is denoted by $N(\fa)$.

Let $L$ be an algebraic number field and
let $S$ be a finite set of places of $L$ which contains all infinite places. Denote by $s$ the cardinality of $S$. Recall that
the ring of $S$-integers $\OO_S$ is defined as
$$
\OO_S=\{ \al\in L \ : \ |\al|_v \leq 1 \ \text{for} \ v\in \M_L \setminus S \}.
$$
If $S$ consists only of the infinite places of $L$, we put $P:=2, Q:=2$. If $S$ contains also finite places, we denote by
$\fp_1, \dots, \fp_t$ the prime ideals corresponding to the finite places of $S$, and we put
$$
P:=\max(N(\fp_1), \dots, N(\fp_t)), \quad\quad Q:=N(\fp_1 \dots \fp_t).
$$
The $S$-regulator associated with $S$ is denoted by $R_S$. If $S$ consists only of the infinite places of $L$ it is just $R_L$, while
otherwise
$$
R_S=h_S R_L \prod_{i=1}^t \log N(\fp_i),
$$
where $h_S$ is a (positive) divisor of $h_L$. It is an easy consequence of formula (2) of Louboutin \cite{Louboutin1} that
\begin{equation}\label{bound_hLRL}
h_L R_L \leq |\Delta_L|^{1/2} (\log^* |\Delta_L|)^{d_L-1};
\end{equation}
cf. formula (59) of Gy\H ory and Yu, \cite{GyYu1}. Further, we have
\begin{equation}\label{R_S}
R_S \leq |\Delta_L|^{1/2} (\log^* |\Delta_L|)^{d_L-1} (\log^* Q)^s;
\end{equation}
see (6.1) in Evertse and Gy\H ory \cite{EGy9}. In view of \eqref{bound_hLRL} this is true also if $t=0$.

\subsubsection{Results in the number field case}
Let $F(X,Y)\in L[X,Y]$ be a binary form of degree $n \geq 3$ with splitting field $L$ and with at least three pairwise non-proportional linear
factors. Further, let $\be \in L \setminus \{ 0 \}$
and consider the Thue equation
\begin{equation}\label{Thue_over_NF}
F(\xi,\eta)=\be \quad\quad \text{in} \quad \xi,\eta\in\OO_S.
\end{equation}
For a polynomial $G$ with algebraic coefficients, we denote by $h(G)$
the maximum of the logarithmic heights of its coefficients.

\begin{proposition}\label{P_Thue_over_NF}
All solutions $(\xi,\eta)\in\OO_S^2$ of equation \eqref{Thue_over_NF} satisfy
\begin{eqnarray}\label{bound_Thue_NF}
&&\max(h(\xi),h(\eta))\leq c_1PR_S\left( 1+(\log^* R_S)/\log^* P \right)\times
\\
\nonumber
&&\qquad\qquad\qquad\qquad\qquad
\times\left( c_2R_L+\frac{h_L}{d_L}\log Q+2nd_L H_1+H_2 \right),
\end{eqnarray}
where
$$
H_1=\max (1,h(F)),\ \ H_2=\max (1,h(\de )),
$$
$$
c_1=250n^6s^{2s+3.5} \cdot 2^{7s+27} (\log 2s)d_L^{2s+4} (\log^*(2d_L))^3
$$
and
$$
c_2=
\begin{cases}
0 & \text{if $r_L=0$} \\
1/d_L & \text{if $r_L=1$} \\
29er_L!r_L\sqrt{r_L-1}\log d_L & \text{if $r_L\geq 2$}. \\
\end{cases}
$$
\end{proposition}

\begin{proof}
This is Corollary 3 of Gy\H ory and Yu \cite{GyYu1}.
\end{proof}

We shall also need the following.

\begin{lemma}\label{L_discriminant_divisibility}
If $L$ is the composite of the algebraic number fields $L_1, \dots, L_k$ with degrees $d_{L_1}, \dots, d_{L_k}$ and
discriminants $\Delta_{L_1}, \dots,  \Delta_{L_k}$, then $\Delta_L$ divides $\Delta_{L_1}^{d_L/d_{L_1}} \dots  \Delta_{L_k}^{d_L/d_{L_k}}$ in $\Z$.
\end{lemma}

\begin{proof}
See Stark \cite{Stark1}.
\end{proof}

\begin{lemma}\label{L_discriminant_extension}
Let $L$ be an algebraic number field and $\theta$ a zero
of a polynomial $G\in L[X]$ of degree $n$ without multiple roots.
Then
\[
|\Delta_{L(\theta )}|\leq n^{(2n-1)d_L}e^{(2n^2-2)h(G)}|\Delta_L|^{[L(\theta ):L]}.
\]
\end{lemma}

\begin{proof}
This is a slight modification of the second assertion of \cite[Lemma 4.1]{BA18}.
In fact, this lemma gives the same bound but with an exponent
$(2n-2)h'(G)$ on $e$, where for $G=\sum_{k=0}^n b_kX^{n-k}$ we define
\[
h'(G)=\sum_{v\in\M_L} d_v\log\max (1,|b_0|_v\kdots |b_n|_v).
\]
This height is easily estimated from above by $\sum_{k=0}^n h(b_k)\leq (n+1)h(G)$.
Our lemma follows.
\end{proof}

\subsubsection{Concluding the proof of Proposition \ref{P_Thue_over_B}}

\begin{proof}[Proof of \eqref{Thue_hbound_B} in Proposition \ref{P_Thue_over_B}]
We first consider the case $q>0$. Let $x,y$ be a solution of \eqref{Thue_over_B} in $B$. We keep the notation introduced in
Section \ref{S_specializations}.
Recall that $\HH:=\Delta_{\F} \cdot \F_D \cdot f$ and by \eqref{bound_deg_H} and
\eqref{bound_d0*_h0*} we get
\begin{equation}\label{bound_deg_H_II_in Thue}
\deg \HH \leq  (nd)^{\exp O(r)}.
\end{equation}

Choose $\uv\in \Z^q$ with $\HH(\uv)\ne 0$, choose $j\in \{ 1, \dots ,D \}$, and denote by $F_{\uv,j}$, $\de^{(j)}(\uv)$, $x^{(j)}(\uv)$, $y^{(j)}(\uv)$, the images of $F,\de,x,y$ under $\vp_{\uv,j}$. Then $F_{\uv,j}$ has its coefficients in $K_{\uv,j}$. Further, let
$L$ denote the splitting field of $F_{\uv,j}$ over $K_{\uv,j}$, and $S$ the set of places of $L$ which consists of all infinite places
and all finite places lying above the rational prime divisors of $f(u)$. Note that $w^{(j)}(\uv)$ is an algebraic integer and
$f(\uv) \in \OO_S^*$. Thus $\vp_{\uv,j}(B)\subseteq \OO_S$ and it follows from \eqref{Thue_over_B} that
\begin{equation}\label{Thue_specialized}
F_{\uv,j}\left(x^{(j)}(\uv),y^{(j)}(\uv) \right)=\de^{(j)}(\uv), \quad\quad x^{(j)}(\uv),y^{(j)}(\uv)\in \OO_S.
\end{equation}

We already proved in Section \ref{S_boundthedegree} that \eqref{Thue_degbound_B} of Proposition \ref{P_Thue_over_B} holds, i.e.
we have
$$
\odeg x, \odeg y \leq (nd)^{\exp O(r)}.
$$
Hence we can apply Lemma \ref{L_bound_hbar by_ha_of conjugates of specializations} with
$$
N=\max\left( (nd)^{\exp O(r)}, 2Dd_0^*+2(q+1)(d_1^*+1)\right).
$$
In view of \eqref{bound_d0*_h0*}, $D\leq d^r$ and $q \leq r$ we get
\begin{equation}\label{bound_Thue_N}
N \leq (nd)^{\exp O(r)}.
\end{equation}
By applying Lemma \ref{L_bound_hbar by_ha_of conjugates of specializations}
with $\al =x$ and $\al =y$, and inserting $D\leq d^r$ and the upper bound
$h_1^*\leq (nd)^{\exp O(r)}(h+1)$ from  \eqref{bound_d0*_h0*},
it follows that there are $\uv\in\Zz^q$, $j\in\{ 1\kdots D\}$ with
\begin{equation}\label{u-bound}
|\uv |\leq (nd)^{\exp O(r)},\ \ \HH (\uv )\ne 0
\end{equation}
and
\begin{eqnarray}\label{hbar-bound}
&&\max (\oh (x),\oh (y))\leq (nd)^{\exp O(r)}\Big[(h+1)^2+
\\
\nonumber
&&\qquad\qquad\qquad\qquad\qquad
+ d^r(h+1)\max \big(h(x^{(j)}(\uv )),h(y^{(j)}(\uv ))\big)\Big].
\end{eqnarray}
We proceed further with this $\uv$, $j$ and apply Proposition \ref{P_Thue_over_NF} to equation \eqref{Thue_specialized} to derive an upper bound for $h(x^{(j)}(\uv))$ and $h(y^{(j)}(\uv))$. To do so we have to bound from above the parameters corresponding to those which occur in Proposition \ref{P_Thue_over_NF}.

Write $F=\sum_{k=0}^n a_kX^{n-k}Y^k$ and put
$$
\odeg F :=\max_{0\leq k\leq n} \odeg a_k,\ \ \oh (F):=\max_{0\leq k\leq n} \oh (a_k).
$$
Notice that by Lemma \ref{L_degbar_hbar_by_d*_h*}, applied
to $\de$ and the coefficients of $F$ with the choice $d^*=d$, $h^*=h$, we have
\begin{align}
\label{bound_degbar_F_degbar_de} \odeg F, \odeg \de &\leq (2d)^{\exp O(r)}, \\
\label{bound_hbar_F_hbar_de} \oh(F), \oh(\de) &\leq (2d)^{\exp O(r)}(h+1).
\end{align}

It follows from Lemma \ref{L_bound_ha_by_degbar_hbar},
$q\leq r$, $D\leq d^r$,
\eqref{bound_d0*_h0*}, \eqref{bound_degbar_F_degbar_de},
\eqref{bound_hbar_F_hbar_de}, and lastly \eqref{u-bound}, that
\begin{eqnarray}\label{Au}
h(F_{\uv,j}) &\leq& D^2+q(D\log d_0^*+\log\odeg F)+D h_0^* +\\
\nonumber
&&\qquad\qquad +\oh(F)+(D d_0^*+\odeg F)\log\max(1,|\uv|)
\\
\nonumber
&\leq& (nd)^{\exp O(r)}(h+1).
\end{eqnarray}
In a similar way, replacing $F$ by $\delta$, we obtain also
\begin{equation}\label{Bu}
h(\de^{(j)}(\uv)) \leq (nd)^{\exp O(r)}(h+1).
\end{equation}

We recall that $d_L$ and $\Delta_L$ denote the degree and the discriminant of $L$ over $\Q$. Since $[K_{\uv,j}:\Q] \leq D$, we
have $d_L\leq Dn!$. Let $G(X):=F(X,1)$,
and let $\theta_1\kdots\theta_{n'}$ be the roots of $G$.
We have $n'=n$ if $a_0\not= 0$ and $n'=n-1$ otherwise.
Then $L=K_{\uv ,j}(\theta_1\kdots\theta_{n'})$.
Denote by $d_{L_i}$ the degree and by $\Delta_{L_i}$ the discriminant of the number field $L_i:=K_{\uv ,j}(\theta_i)$,
$i=1, \dots, n'$. Then by Lemma \ref{L_discriminant_divisibility} we have
\begin{equation}\label{bound_Delta_L}
|\Delta_L|\leq \prod_{i=1}^{n'} |\Delta_{L_i}|^{d_L/d_{L_i}}.
\end{equation}
We estimate $|\Delta_L|$. First notice that by Lemma \ref{L_bound_fielddisc_spec}, inserting the estimates $q\leq r$, $D\leq d^r$,
\eqref{bound_d0*_h0*}, \eqref{u-bound},
\begin{eqnarray}\label{DiscKuj-bound}
|\Delta_{K_{\uv,j}}|&\leq&
D^{2D-1}\big( (d_0^*)^qe^{h_0^*}\max(1,|\uv |^{d_0^*}|)\big)^{2D-2}
\\
\nonumber
&\leq& \exp \big( (nd)^{\exp O(r)}(h+1)\big).
\end{eqnarray}
Further, by Lemma \ref{L_discriminant_extension}
and the estimates $D\leq d^r$, \eqref{Au}, \eqref{DiscKuj-bound},
\[
\begin{aligned}
|\Delta_{L_i}|&\leq n^{(2n-1)D}e^{(2n^2-2)h(F_{\uv ,j})}
|\Delta_{K_{\uv ,j}}|^{[L_i:K_{\uv,j}]}
\\
&\leq \exp\{ [L_i:K_{\uv ,j}]\cdot (nd)^{\exp O(r)}(h+1)\}.
\end{aligned}
\]
By inserting this into \eqref{bound_Delta_L}, using $[L:K_{\uv ,j}]\leq n!$,
we obtain
\begin{eqnarray}\label{bound_Delta_L_II}
|\Delta_L| &\leq&
\exp \left\{ (nd)^{\exp O(r)}(h+1)\cdot nd_L/d_{K_{\uv ,j}}\right\}
\\
\nonumber
&\leq&\exp\{ n! (nd)^{\exp O(r)}(h+1)\}.
\end{eqnarray}

By assumption \eqref{define d0*_d1*_h0*_h1*}, $f$ has degree at most $d_1^*$ and logarithmic height at most $h_1^*$.
Further, $f(\uv)\ne 0$ and by $q\leq r$,
\eqref{bound_d0*_h0*}, \eqref{u-bound},
\begin{equation}\label{bound_hfu}
|f(\uv)|\leq (d_1^*)^q e^{h_1^*}\max (1,|\uv|)^{d_1^*}
\leq \exp\{ (nd)^{\exp O(r)}(h+1)\}.
\end{equation}
The cardinality $s$ of $S$ is at most $d_L(1+\omega)$, where $\omega$ denotes the number of distinct prime divisors of $f(\uv)$.
By prime number theory,
\begin{equation}\label{bound_Thue s}
s= O (d_L\log^* |f(\uv )|/\log^*\log^* |f(\uv )|).
\end{equation}
From this estimate and \eqref{bound_hfu},
$D\leq d^r$, $d_L\leq n!d^r$,
one easily deduces that for $c_1$ coming from Proposition \ref{P_Thue_over_NF} we have
\begin{equation}\label{bound_c1}
c_1 \leq \exp\{ n! (nd)^{\exp O(r)}(h+1)\}.
\end{equation}
Next, we estimate $P,Q$ and $R_S$. By \eqref{bound_hfu}, $d_L\leq n!d^r$ we have
\begin{equation}\label{bound_Thue_P_Q}
P\leq Q \leq |f(\uv)|^{d_L} \leq \exp\{ n!(nd)^{\exp O(r)}(h+1)\}.
\end{equation}
To estimate $R_S$, we use \eqref{R_S}. Then, in view of \eqref{bound_Delta_L_II} and $d_L\leq n!d^r$, we have
\begin{equation}\label{bound_Thue Delta_L_II}
|\Delta_L|^{1/2}(\log^* |\Delta_L|)^{d_L-1}\leq \exp\{ n!(nd)^{\exp O(r)}(h+1)\}.
\end{equation}
Further, by \eqref{bound_Thue s} and \eqref{bound_Thue_P_Q},
\[
(\log Q)^s\leq \exp\left\{ O\Big( d_L\frac{\log^* |f(\uv )|}{\log^*\log^* |f(\uv )|}
\cdot (\log d_L+\log^*\log^* |f(\uv )|)\Big)\right\}.
\]
Together with \eqref{bound_hfu},
this leads to
\begin{equation}\label{bound_RS}
R_S \leq |\Delta_L|^{1/2}(\log^* |\Delta_L|)^{d_L-1}(\log Q)^s\leq
\exp\{ n!(nd)^{\exp O(r)}(h+1)\}.
\end{equation}
Combining \eqref{bound_hLRL} with \eqref{bound_Thue Delta_L_II} and with $R_L > 0.2052$ (see Friedman \cite{Friedman1}) we get
\begin{equation}\label{bound_hLRL_2}
\max(h_L, R_L) \leq \exp\{ n!(nd)^{\exp O(r)}(h+1)\}.
\end{equation}
Finally, using $r_L<d_L\leq n!d^r$, we infer that
\begin{equation}\label{bound_c2}
c_2 \leq \exp O(d_L \log^* d_L)\leq \exp\{ n!(nd)^{\exp O(r)}\}.
\end{equation}
We now apply Proposition \ref{P_Thue_over_NF} to
equation \eqref{Thue_specialized}.
From the estimates \eqref{Au}, \eqref{Bu}, \eqref{bound_c1}, \eqref{bound_Thue_P_Q}, \eqref{bound_RS}, \eqref{bound_hLRL_2},
\eqref{bound_c2}, it follows that the upper bound in
Proposition \ref{P_Thue_over_NF} is a sum and product of terms,
which are all bounded above by $\exp\{ n! (nd)^{\exp O(r)}(h+1)\}$.
It follows that
\[
h\left(x^{(j)}(\uv)\right),h\left(y^{(j)}(\uv) \right)
\leq \exp\{ n! (nd)^{\exp O(r)}(h+1)\}.
\]
By inserting this into \eqref{hbar-bound}, we obtain
the upper bound \eqref{Thue_hbound_B} in Proposition \ref{P_Thue_over_B} for $q>0$.

Now assume $q=0$. In this case $K_0=\Q$, $A_0=\Z$ and $B=\Z[f^{-1},w]$, where $w$ is an algebraic integer with minimal polynomial
$\F(X)=X^D+\F_1X^{D-1}+\dots +\F_D\in \Z[X]$ over $\Q$, and $f$ is a non-zero rational integer. In view of Propositions \ref{P_newdomaingenerators} (i) and \ref{P_newdomaingenerators_spec_case}
we may assume that
$$
\log |f| \leq h_1^* \quad\quad \text{and} \quad\quad \log|\F_k| \leq h_0^* \quad\quad \text{for} \quad k=1, \dots, D,
$$
where $h_0^*, h_1^*$ satisfy \eqref{bound_d0*_h0*}. Denote by $w^{(1)}, \dots, w^{(D)}$ the conjugates of $w$, and let $K_j:=\Q(w^{(j)})$ for
$j:=1\kdots D$.
By a similar argument as in the proof of Lemma 5.5 of Evertse and Gy\H ory \cite{EGy9}, we have $|\Delta_{K_j}|\leq D^{2D-1}e^{(2D-2)h_0^*}$,
which is the estimate from Lemma \ref{L_bound_fielddisc_spec} with $q=0$
and $\max (1,|\uv |)$ replaced by $1$.
For $\al\in K$, we denote by $\al^{(j)}$ the conjugate
of $\al$ corresponding to $w^{(j)}$.

Instead of Lemma \ref{L_bound_hbar by_ha_of conjugates of specializations}
we use Lemma \ref{L_bound_hbar_for_q_eq_0}, applied with $G=\F$, $m=D$ and $\be^{(j)}=x^{(j)}$, resp. $y^{(j)}$. Inserting \eqref{bound_d0*_h0*},
this leads to an estimate
\begin{equation}\label{hbar-bound q=0}
\max (\oh (x),\oh (y))\leq (nd)^{\exp O(r)}\max_{1\leq j\leq D}
\max \big( h(x^{(j)}),h(y^{(j)})\big).
\end{equation}
We proceed further with the $j$ for which the maximum is assumed.

Now we can follow the argument for the case $q>0$,
except that in all estimates we have to take $q=0$,
and replace
$\max (1,|\uv |)$ by $1$,
$K_{\uv ,j}$ by $K_j$, $f(\uv )$ by $f$, $F_{\uv ,j}$ by $F^{(j)}$,
where $F^{(j)}$ is the binary form obtained by taking the $j$-th conjugates
of the coefficients of $F$, and $f(\uv )$ by $f$.
This leads to an estimate
\[
h((x^{(j)})),h((y^{(j)})
\leq \exp\{ n! (nd)^{\exp O(r)}(h+1)\},
\]
and combined with
\eqref{hbar-bound q=0} this gives again
\eqref{Thue_hbound_B}.
This completes the proof of Proposition \ref{P_Thue_over_B}.
\end{proof}

\subsection{Hyper- and superelliptic equations}

\subsubsection{Results in the number field case.}\label{S_NFcase}

Let $L$ be a number field, and denote as usual
by $d_L$, $\Delta_L$, $\OO_L$, $\M_L$ its degree, discriminant,
class number, regulator, ring of integers, and set of places.
Further, let $S$ be a finite set of places of $L$ containing all infinite places.
If $S$ consists only of the infinite places of $L$, put $P:=2, Q:=2$.
Otherwise,
denote by
$\fp_1, \dots, \fp_t$ the prime ideals corresponding to the finite places of $S$, and put
$$
P:=\max(N(\fp_1), \dots, N(\fp_t)), \quad\quad Q:=N(\fp_1 \dots \fp_t).
$$

Let
\begin{equation}\label{poly}
F(X)= a_0X^n+a_1X^{n-1}+\dots+a_n\in\OO_S[X]
\end{equation}
be a polynomial of degree $n\geq 2$ and of non-zero discriminant,
$\de\in \OO_S\setminus\{ 0\}$,
and $m$ a positive integer. Put
$$
\hh:=\sum_{v \in \M_L} d_v\log \max (1,|\de|_v,|a_0|_v, \dots ,|a_n|_v),
$$
where $d_v:=[L_v:\Qq_p]/[L:\Qq ]$, with $p\in\M_{\Qq}$ the place below $v$.

\begin{proposition}\label{P_super_NF}
Assume $n\geq 2$, $m\geq 3$. If $x,y\in \OO_S$ is a solution to the equation
\begin{equation}\label{super}
F(x)=\de y^m, \quad \quad \mbox{$x,y\in \OO_S$,}
\end{equation}
then
\begin{equation*}
h(x),h(y)  \leq c_3^{m^3}
|\Delta_L|^{2m^2n^2}Q^{3m^2n^2}e^{8m^2n^3d_L\widehat{h}},
\end{equation*}
where $c_3 := (6ns)^{14n^3s}$.
\end{proposition}

\begin{proof}
This is Theorem 2.1 in \cite{BA18}.
\end{proof}

\begin{proposition}\label{P_hyper_NF}
Let $n\geq 3$. If $x,y\in \OO_S$ is a solution to
\begin{equation}\label{hyper}
F(x)=\de y^2, \quad \quad \mbox{$x,y\in \OO_S$,}
\end{equation}
then
\begin{equation*}
h(x), h(y) \leq
c_4 |\Delta_L|^{8n^3}Q^{20n^3}e^{50n^4d_L\widehat{h}},
\end{equation*}
where $c_4 :=(4ns)^{212n^4s}$.
\end{proposition}

\begin{proof}
This is Theorem 2.2 in \cite{BA18}.
\end{proof}

\begin{proposition}\label{P_ST}
Let $n\geq 2$.
If $x,y,m$ is a solution to
\begin{equation*}
F(x)=\de y^m, \quad\quad\mbox{$x,y\in\OO_S$, $m\in\Zz_{\geq 2}$,}
\end{equation*}
such that $y\not= 0$ and $y$ is not a root of unity, then
\begin{equation*}
m \leq c_5 |\Delta_L|^{6n}P^{n^2}e^{11nd_L\widehat{h}},
\end{equation*}
where $c_5:= (10n^2s)^{40ns}$.
\end{proposition}

\begin{proof}
This is Theorem 2.3 in \cite{BA18}.
\end{proof}

\subsubsection{Concluding the proofs of Propositions \ref{P_hypersuper_over_B}
and \ref{P_ST_over_B}}

\begin{proof}[Proof of \eqref{hypersuper_hbound_B} in Proposition \ref{P_hypersuper_over_B}]
The computations will be similar to those in the proof of
\eqref{Thue_hbound_B} in Proposition \ref{P_Thue_over_B} but with some
simplifications.

First we suppose $q>0$.
Take a solution $x,y$ of \eqref{hypersuper_over_B} in $B$. We use again the polynomial
$\HH:=\Delta_{\F} \cdot \F_D \cdot f$ from
Section \ref{S_specializations}.
Take again $\uv\in\Zz^q$ with $\HH (\uv )\ne 0$,
choose $j\in \{ 1, \dots ,D \}$, and denote by $F_{\uv,j}$, $\de^{(j)}(\uv)$, $x^{(j)}(\uv)$, $y^{(j)}(\uv)$, the images of $F,\de,x,y$ under
the specialization $\vp_{\uv,j}$.
In contrast to our argument for Thue equations, we do not have to deal with
the splitting field of $F$ now. So we take for $S$ the set of places
of $K_{\uv ,j}$, consisting of all infinite places,
and all finite places lying above the rational prime divisors of $f(\uv )$.
Then $\vp_{\uv ,j}(B)\subseteq \OO_S$, and
\begin{equation}\label{hypersuper_specialized}
F_{\uv ,j}(x^{(j)}(\uv )) =\de^{(j)}(\uv )y^{(j)}(\uv )^m,
\ \ \x^{(j)}(\uv ),\, y^{(j)}(\uv )\in\OO_S.
\end{equation}
Note that by
the choice of $\HH$ and $\HH(\uv)\ne 0$ we have
$\de_j(\uv )\ne 0$ and $F_{\uv ,j}$ has non-zero discriminant.
So $F_{\uv ,j}$ has the same number of zeros and degree as $F$,
that is, the degree of $F_{\uv ,j}$ is $n\geq 2$ if $m\geq 3$ and $n\geq 3$
if $m=2$. Hence Propositions \ref{P_super_NF} and \ref{P_hyper_NF}
are applicable.

By precisely the same argument as in the case for Thue equations,
there are $\uv\in\Zz^q$ and $j\in\{ 1\kdots D\}$ with
\eqref{u-bound} and \eqref{hbar-bound}. We proceed further with this $\uv$, $j$.

We estimate the parameters corresponding to those in the bounds
from Propositions \ref{P_super_NF}, \ref{P_hyper_NF}.
First, we get precisely the same estimates as in \eqref{Au} and \eqref{Bu}.
These imply
\begin{equation}\label{hhat-bound}
\widehat{h}\leq (n+1)h(F_{\uv ,j})+h(\de^{(j)}(\uv ))\leq (nd)^{\exp O(r)}(h+1).
\end{equation}
Further we have, similarly to \eqref{DiscKuj-bound},
\begin{equation}\label{DiscKuj_super}
|\Delta_{K_{\uv ,j}}|\leq \exp\{ (nd)^{\exp O(r)}(h+1)\}.
\end{equation}
Next, similar to \eqref{bound_hfu},
\begin{equation}\label{bound_hfu_super}
|f(\uv )|\leq \exp\{ (nd)^{\exp O(r)}(h+1)\}.
\end{equation}
The set $S$ now consists of places of $K_{\uv ,j}$ instead
of the splitting field of $F_{\uv ,j}$ over $K$.
So since $[K_{\uv ,j}:\Qq ] \leq D$ we now have $s\leq D(1+\omega )$,
where $\omega$ is the number of distinct prime divisors of $f(\uv )$.
This gives, instead of \eqref{bound_Thue s},
\begin{equation}\label{bound_super s}
s=O\left( D\log^* |f(\uv )|/\log^*\log^* |f(\uv )|\right).
\end{equation}
By inserting \eqref{bound_hfu_super}, and $D\leq d^r$, we obtain
for the quantities $c_3,c_4$ in Propositions \ref{P_super_NF}
and \ref{P_hyper_NF} the upper bounds
\begin{equation}\label{c3-c4}
c_3,\, c_4 \, \leq \exp\{ (nd)^{\exp O(r)}(h+1)\}.
\end{equation}
Lastly, we have instead of \eqref{bound_Thue_P_Q},
\begin{equation}\label{bound_super_P_Q}
P\leq Q\leq |f(\uv )|^D\leq \exp\{ (nd)^{\exp O(r)}(h+1)\},
\end{equation}
where we have used \eqref{bound_hfu_super} and $D\leq d^r$.

We now apply Propositions \ref{P_super_NF} and \ref{P_hyper_NF}
to \eqref{hypersuper_specialized}. Note that we have to take
$L=K_{\uv ,j}$; so $d_L\leq D\leq d^r$. By inserting this and
\eqref{hhat-bound}, \eqref{DiscKuj_super},
\eqref{c3-c4}, \eqref{bound_super_P_Q} into the upper bounds
from these Propositions, we obtain
\begin{equation}\label{xuj-yuj-super}
h(x^{(j)}(\uv )),\, h(y^{(j)}(\uv ))\, \leq \exp\{ m^3 (nd)^{\exp O(r)}(h+1)\}.
\end{equation}
By inserting this into \eqref{hbar-bound},
we obtain \eqref{hypersuper_hbound_B} in the case $q>0$.

Now let $q=0$. For $\alpha\in K$, write $\alpha^{(j)}$ for the conjugate
of $\alpha$ corresponding to $w^{(j)}$, and let $F^{(j)}$ be the
polynomial obtained by taking the $j$-th conjugates of the
coefficients of $F$. We simply have to follow the above arguments,
replacing everywhere $q$ by $0$, $\max (1,|\uv |)$ by $1$,
$K_{\uv ,j}$ by $K^{(j)}=\Qq (w^{(j)})$, $F_{\uv ,j}$ by $F^{(j)}$,
$x^{(j)}(\uv )$, $y^{(j)}(\uv )$ by $x^{(j)}$, $y^{(j)}$,
and $f(\uv )$ by $f\in\Zz$. Instead of \eqref{hbar-bound}
we have to use \eqref{hbar-bound q=0}. Thus, we obtain
the same estimate as \eqref{xuj-yuj-super},
but with $x^{(j)}$, $y^{(j)}$ instead of $x_j(\uv )$, $y_j(\uv )$.
Via \eqref{hbar-bound q=0} we obtain
\eqref{hypersuper_hbound_B} in the case $q=0$.
This completes our proof of Proposition
\ref{P_hypersuper_over_B}.
\end{proof}

\begin{proof}[Proof of Proposition \ref{P_ST_over_B}]
Assume for the moment $q>0$.
Let $x\in B$, $y\in B\cap\oQ$, $m\in\Zz_{\geq 2}$ be a solution
of \eqref{hypersuper_over_B},
such that $y\not= 0$ and $y$ is not a root of unity.
Choose again $\uv$, $j$
with \eqref{u-bound}, \eqref{hbar-bound}.
Note that $y^{(j)}(\uv )$ is a conjugate of $y$ since $y\in\oQ$;
hence it is not $0$ or a root of unity.

We apply Proposition \ref{P_ST} to \eqref{hypersuper_specialized}.
By \eqref{bound_hfu_super}, \eqref{bound_super s},
we have for the constant $c_5$ in Proposition \ref{P_ST}, that
\[
c_5\leq \exp\{ (nd)^{\exp O(r)}(h+1)\}.
\]
Further, we have the upper bounds \eqref{hhat-bound} for
$\widehat{h}$, \eqref{DiscKuj_super} for $|\Delta_{K_{\uv ,j}}|$,
and \eqref{bound_super_P_Q} for $P$. By inserting these estimates
into the upper bound for $m$ from Proposition \ref{P_ST},
we obtain $m\leq \exp\{ (nd)^{\exp O(r)}(h+1)\}$.
In the case $q=0$, we obtain the same estimate,
by making the same modifications as in the proof
of Proposition \ref{P_hypersuper_over_B}.
This finishes our proof of Proposition \ref{P_ST_over_B}.
\end{proof}

\end{document}